\documentclass[reqno,11pt]{amsart} 
    \usepackage{amsmath,amscd,amsfonts,amssymb}
    \usepackage{mathrsfs,dsfont}
    \usepackage{color}
    \usepackage{mathtools}
    \usepackage{hyperref}
    \usepackage{tikz-cd}
    \usepackage[normalem]{ulem}
    \usepackage{xcolor}
    \usepackage{appendix}
    \usepackage{amsthm}
    \usepackage{thmtools}

    \usepackage{tikz}
    \usetikzlibrary{decorations.pathreplacing}
    
    \numberwithin{equation}{section}
    \numberwithin{figure}{section}
    
    \def\R{\mathbb{R}}

    \def\T{\mathbb{T}}

    \def\eps{\varepsilon}

    \renewcommand\leq{\leqslant}
    \renewcommand\geq{\geqslant}

    \newcommand{\PW}{\mathrm{PW}}

    \theoremstyle{plain}
    \newtheorem{thm}{Theorem}[section]
    \newtheorem{theorem}[thm]{Theorem}
    
    \newtheorem{lemma}[thm]{Lemma}

    \newtheorem{proposition}[thm]{Proposition}
    
    \newtheorem{question}[thm]{Question}
    
    \newtheorem{conjecture}[thm]{Conjecture}
    
    \newtheorem*{claim*}{Claim}

    \theoremstyle{definition}
    \newtheorem{definition}[thm]{Definition}
    \newtheorem*{definition*}{Definition}
    \newtheorem*{remarks*}{Remarks}
    \newtheorem*{remark*}{Remark}
    \newtheorem{remark}[thm]{Remark}
    \newtheorem{example}[thm]{Example}

    \newcommand{\V}{\mathcal V}
    \newcommand{\SK}{\mathrm{SK}}
    
    \newcommand{\rest}{\mathrm{rest}}
    \newcommand{\BR}{\mathrm{BR}}
    \newcommand{\x}{\mathbf{x}}
    \newcommand{\dir}{\mathrm{dir}}
    \renewcommand{\P}{\mathrm{P}}
    \renewcommand{\H}{\mathrm{H}}
    \newcommand{\lin}{\mathrm{lin}}

\begin{document}

	\title{Bourgain's Condition, Sticky Kakeya, and new Examples}
 
	\author{Arian Nadjimzadah}
	\address{UCLA Department of Mathematics, Los Angeles, CA 90095.}
	\email{anad@math.ucla.edu}

	\date{}
	

    \begin{abstract}
        We prove that in all dimensions at least 3 and for any H\"ormander-type phase function satisfying Bourgain's condition, the sticky case of the corresponding curved Kakeya conjecture reduces to the sticky case of the classical Kakeya conjecture. This supports a conjecture of Guo--Wang--Zhang that an oscillatory integral operator satisfies the same $L^p$ bounds as in the restriction conjecture exactly when its phase function satisfies Bourgain's condition. 

        Our result follows from a new geometric characterization of Bourgain's condition in terms of straightening curved $\delta$-tubes in a $\delta^{1/2}$-tube.
        We construct examples in all dimensions at least 3 which show this local straightening property does not persist in a larger tube and, in particular, these are the first phase functions satisfying Bourgain's condition for which there is no diffeomorphism taking the corresponding families of curves to lines. This suggests that a general to sticky reduction in the spirit of Wang--Zahl needs substantial new ideas, and we take initial steps in this direction. We expect these examples to serve as a natural testing ground.  
    \end{abstract}
	
	\maketitle

\section{Introduction}
We recall the definition of a H\"ormander-type phase function and the corresponding H\"ormander-type oscillatory integral operator. Fix $n \geq 2$ and let $\Sigma \subset \R^{n-1}$ and $M \subset \R^{n}$ be open balls, with origins $0_M \in M$ and $0_\Sigma \in \Sigma$. Throughout the paper, we may restrict $M$ and $\Sigma$ to smaller open balls containing $0_M$ and $0_\Sigma$.
We say that a smooth function $\phi : M \times \Sigma \to \R$ is a \emph{H\"ormander-type phase function in $\R^n$} if it satisfies the following conditions: 
\begin{enumerate}
    \item[(H1)] $\mathrm{rank}(\nabla_\x \nabla_\xi \phi(\x,\xi)) = n-1$ on $M \times \Sigma$, and 
    \item[(H2)] with the Gauss map $G : M \times \Sigma \to \R^n \setminus \{0\}$ defined by 
    \begin{align}
        G(\x,\xi) = \bigwedge_{j =1}^{n-1} \partial_{\xi_j} \nabla_\x \phi(\x,\xi),
    \end{align}
    the matrix 
    \begin{align} \label{eq: H2}
        (G(\x,\xi) \cdot \nabla_\x)\nabla_\xi^2 \phi(\x,\xi)
    \end{align}
    is nondegenerate on $M \times \Sigma$. 
\end{enumerate} 
We say that (H2+) holds if the matrix \eqref{eq: H2} is further positive-definite, and we call $\phi$ a \emph{positive-definite} H\"ormander-type phase function in that case.
Let $a : \R^{n} \times \R^{n-1} \to \R$ be a smooth bump function supported on $M \times \Sigma$. We define the H\"ormander-type oscillatory integral operator associated to $\phi$ by 
\begin{align}
    T_\phi^\lambda f(\x) = \int_{\R^{n-1}} e^{i \lambda \phi(\x/\lambda,\xi)} f(\xi) a(\x/\lambda,\xi) d\xi.
\end{align}
A central open problem is to determine the range of $p$ for which 
\begin{align}\label{eq: op boundedness}
    \|T_\phi^\lambda f\|_p \leq C_{p,\eps} \lambda^\eps \|f \|_\infty.
\end{align}
This was posed by H\"ormander in the 1970s and is often called H\"ormander's oscillatory integral problem \cite{HormanderOriginal}. See also \cite{HormanderDichotomy} for an exposition of this problem. Writing $\x = (x,t) \in \R^{n-1} \times \R$ and defining 
\begin{align}
    \phi_{n,\rest}(\x,\xi) = x \cdot \xi + \tfrac{1}{2}t|\xi|^2, \quad (0_M, 0_\Sigma) = (0,0),
\end{align}
$T_{\phi_{n,\rest}}^\lambda$ is the Fourier extension operator for the paraboloid. In that case we obtain the Fourier restriction problem in dimension $n$, and the conjectured range for \eqref{eq: op boundedness} is $p > \frac{2n}{n-1}$. 
If \eqref{eq: op boundedness} holds for the phase function
\begin{align}
    \phi_{n,\BR}(\x,\xi) = t^{-1} \sqrt{1 + |x - t \xi|^2}, \quad (0_M,0_\Sigma)=((0,1),0)
\end{align}
for all $p > \tfrac{2n}{n-1}$,
it is known that the Bochner--Riesz conjecture holds in dimension $n$. See \cite[Section 2]{BRoldapproachrevisited} for an explanation of the Bochner--Riesz conjecture and this reduction. 
These examples begin to demonstrate the importance of the following more refined question.
\begin{question}\label{question: best-case Hormander}
    For which H\"ormander-type phase functions $\phi$ does $T_\phi^\lambda$ satisfy \eqref{eq: op boundedness} for all $p > \frac{2n}{n-1}$?
\end{question}

We will often write \emph{phase function} to mean H\"ormander-type phase function. 
As shown by H\"ormander, the answer to Question \ref{question: best-case Hormander} is all phase functions in the $n = 2$ case \cite{HormanderOriginal}. However Bourgain proved that \eqref{eq: op boundedness} fails for a \emph{generic} phase function in the $n = 3$ case \cite{BourgainSeveralVariables}. This construction was later generalized to all $n \geq 3$ by Guo--Wang--Zhang \cite{HormanderDichotomy}: if $\phi$ is a phase function such that for some $(\x,\xi) \in M \times \Sigma$,
\begin{align}\label{eq: failing bourgain}
    (G(\x,\xi) \cdot \nabla_\x)^2 \nabla_\xi^2\phi(\x,\xi) \text{ is not a multiple of } (G(\x,\xi) \cdot \nabla_\x) \nabla_\xi^2 \phi(\x, \xi),
\end{align}
then \eqref{eq: op boundedness} fails for all 
\begin{align}
    p < \frac{2n}{n-1} + \frac{2}{(n-1)(2n^2 - n - 2)}.
\end{align}
Going forward we assume that $n \geq 3$. Motivated by these examples, Guo--Wang--Zhang introduced \emph{Bourgain's condition}. 
\begin{definition}[Bourgain's condition]\label{def: Bourgain cond original}
    Let $\phi$ be a H\"ormander-type phase function in $\R^n$. Then $\phi$ satisfies Bourgain's condition if there exists a (necessarily smooth) function $\lambda : M \times \Sigma \to \R$ such that on $M \times \Sigma$,
    \begin{align}\label{eq: B cond original}
        (G(\x,\xi) \cdot \nabla_\x)^2 \nabla_\xi^2 \phi(\x,\xi) = \lambda(\x,\xi) (G(\x, \xi) \cdot \nabla_\x) \nabla_\xi^2 \phi(\x,\xi). 
    \end{align}
\end{definition}
The phase functions $\phi_{n,\rest}$ and $\phi_{n, \BR}$ notably satisfy Bourgain's condition, and \eqref{eq: failing bourgain} holding at some point $(\x,\xi)$ is the statement that Bourgain's condition fails. This leads to the following conjecture proposed by Guo--Wang--Zhang to hopefully answer Question \ref{question: best-case Hormander}, and unify the Bochner--Riesz and Fourier restriction conjectures. 
\begin{conjecture}[H\"ormander dichotomy conjecture]\label{conj: Hormander dichotomy}
    The operators $T_\phi^\lambda$ satisfy \eqref{eq: op boundedness} for $p > \frac{2n}{n-1}$ if and only if $\phi$ satisfies Bourgain's condition.  
\end{conjecture}
Guo--Wang--Zhang proved, via an involved calculation, that Bourgain's condition is invariant under diffeomorphism in $\x$ and $\xi$ separately---this is promising since the estimate \eqref{eq: op boundedness} is also invariant under such diffeomorphisms. Finding a geometric explanation for this invariance is one of the goals of our work. 
In support of Conjecture \ref{conj: Hormander dichotomy}, Guo--Wang--Zhang proved strong oscillatory integral estimates for phase functions satisfying Bourgain's condition, using a polynomial partitioning argument. Precisely, if $\phi$ satisfies Bourgain's condition and (H2+), then \eqref{eq: op boundedness} holds for 
\begin{align}
    p > 2 + \frac{2.5921}{n} + O(n^{-2}).
\end{align}
At the time of its publication, this gave the best bounds for the restriction conjecture in high dimensions. Both this positive result when Bourgain's condition is satisfied, and the negative result when Bourgain's condition fails, rely on the geometry of the wavepackets of $T^\lambda_\phi$. In other words, they rely on the properties of certain Kakeya sets of curves. 

\subsection{The wavepacket geometry of $T^\lambda_\phi$ and Kakeya sets of curves}

\begin{definition}[$\phi$-curves and $(\delta,\phi)$-tubes]
    Let $\phi$ be a H\"ormander-type phase function. 
    \begin{itemize}
        \item  We define 
    a $\phi$-curve $\ell_{\xi,v}$ by
    \begin{align}
        \ell_{\xi,v} = \{\x \in M :\nabla_\xi \phi(\x,\xi) = v\}.
    \end{align}
    After possibly shrinking $M$ and $\Sigma$, the condition (H1) guarantees that $\ell_{\xi,v}$ is a curve for $(\xi,v) \in \Sigma \times \V$, where $\V := \nabla_\xi\phi(M,\Sigma)\subset \R^{n-1}$ is an open set. We also define the origin $0_\V := \nabla_\xi \phi(0_M, 0_\Sigma)\in \V$, which has the property that $0_M \in \ell_{0_\Sigma, 0_\V}$. We define the $2(n-1)$-parameter family of $\phi$-curves 
    \begin{align}
        \mathcal C(\phi) := \{\ell_{\xi,v} : (\xi,v) \in \Sigma \times \V\}.
    \end{align}
    \item A $(\delta,\phi)$-tube $T_{\xi,v}^\delta \subset M$ is the $\delta$-neighborhood of the $\phi$-curve $\ell_{\xi,v}$. This curve is called the \emph{coaxial curve} of $T_{\xi,v}^\delta$. When the parameter $\delta$ is clear from context, we call these $\phi$-tubes. 
    \end{itemize}
\end{definition}

The $\phi_{n,\rest}$-curves are the lines $\mathrm{line}_{\xi,v} := (v,0) - \R(\xi,1)$. We will always refer to the $\phi_{n,\rest}$-curve with parameter $(\xi,v)$ by $\mathrm{line}_{\xi,v}$. We will call $\phi_{n,\rest}$-tubes \emph{straight tubes} and write them as $T_{\mathrm{line},\xi,v}^{\delta}$.
In analogy with this special case, the parameter $\xi \in \Sigma$ will be called the \emph{direction}. As a consequence of (H2), the parameter $\xi$ is really the correct generalization of direction: if $\ell_{\xi,v}$ and $\ell_{\xi',v'}$ intersect at a point, then $\angle(\ell_{\xi, v}, \ell_{\xi', v'}) \sim |\xi - \xi'|$. Given $\ell = \ell_{\xi,v}$, we define $\dir(\ell) = \xi$.

The counterexamples of Bourgain and Guo--Wang--Zhang arose from special configurations of $\phi$-tubes, parts of which can compress into a surface. On the other hand when $\phi$ satisfies Bourgain's condition, Guo--Wang--Zhang proved that direction separated $\phi$-tubes satisfy a strong form of the \emph{polynomial Wolff axioms} \cite[Theorem 6.2]{HormanderDichotomy}. This is the only part of the proof of their positive result for $T_\phi^\lambda$ \cite[Theorem 1.3]{HormanderDichotomy} that used Bourgain's condition.
One can use \cite[Theorem 3.1]{HormanderDichotomy} to give the following illustrative version of polynomial Wolff axioms for direction separated $\phi$-tubes. We will return to this in Section \ref{subsec: future work} and give a proof in Appendix \ref{subsec: dir sep implies PWA}.

\begin{restatable}[Polynomial Wolff axioms for $\phi$-tubes]{theorem}{thmdirsepimpliesPWA} \label{thm: dir sep implies PWA}
    Let $\phi$ be a phase function in $\R^n$ satisfying Bourgain's condition. For every $\eps > 0$ and integer $E \geq 2$, there exists $C(n, E,\eps) > 0$ such that for any $\lambda \in (\delta,1]$ and every $(\delta/\lambda)$-direction separated collection of $(\delta,\phi)$-tubes $\T$,
    \begin{align}
        |\{T \in \T : T\cap B_{\lambda} \neq \emptyset \text{ and } T \cap 2 B_{\lambda} \subset S\}| \leq C(n,E,\eps)\delta^{-\eps} |S|  (\delta^{n-1}  \lambda)^{-1},
    \end{align}
    whenever $S \subset 4B_{\lambda}$ is a semi-algebraic set of complexity $\leq E$, and $B_\lambda$ is a ball of radius $\lambda$ (and $aB_\lambda$ is the $a$-dilate of $B_\lambda$ around its center). 
\end{restatable}

More recently, Wang--Wu put forward a powerful technique to convert Kakeya-type incidence estimates to restriction type estimates, which has led to the state of the art for the restriction conjecture in all dimensions \cite{wangwurestriction}. To reach the full restriction conjecture, one likely needs a more complete understanding of the incidence geometry of straight tubes, and a first step is to understand the Kakeya conjecture---a compact subset of $\R^n$ containing a line segment in every direction has Hausdorff dimension $n$. Wang--Zahl have recently made spectacular progress in this direction by solving the $n = 3$ case of the Kakeya conjecture \cite{wangZahlstickyKakeya,WangZahlAssouadKakeya,wangZahlKakeya} (see \cite{guth2026streamlinedproofkakeyaset} for a streamlined version of the proof). Our understanding of H\"ormander-type oscillatory integrals in dimension $n \geq 3$ follows the same paradigm: one can often convert progress on certain Kakeya-type problems for curves associated to $\phi$ to progress on H\"ormander's oscillatory integral problem (see \cite{HormanderDichotomy, beyondUniversalEstimates, BourgainSeveralVariables, DaiOscillatory, GuthSharpOscillatory} for examples).

To avoid any artificial issues in studying a Kakeya set whose curves lie near (resp. the parameters of whose curves lie near) the boundary of $M$ (resp. $\Sigma$ or $\V$), we fix $M_0 \subset M$ and $\Sigma_0 \subset \Sigma$ closed balls containing $0_M$ and $0_\Sigma$ respectively. We also set $\V_0 = \nabla_\xi \phi(M_0, \Sigma_0)$. 

\begin{definition}[$\phi$-Kakeya set]
    A compact set $K \subset M_0$ is a $\phi$-Kakeya set if for each $\xi \in \Sigma_0$, there is a $v \in \V_0$ such that $\ell_{\xi, v} \cap M_0 \subset K$. 
\end{definition}
The $\phi_{n,\rest}$-Kakeya sets are standard Kakeya sets in $\R^n$, after harmlessly taking a union with $O(1)$ many rotated copies of itself. It is a classical consequence of Khintchine's inequality that $\phi$-Kakeya sets have dimension $n$ if \eqref{eq: op boundedness} holds for $p > \frac{2n}{n-1}$ (one may follow the arguments in \cite[Proposition 5.7]{DemeterBook} to see this). 
Given Conjecture \ref{conj: Hormander dichotomy}, we are led to the following conjecture for $\phi$-Kakeya sets.

\begin{conjecture}[Curved Kakeya conjecture for Bourgain's condition]\label{conj: Bourgain Kakeya}
    If $\phi$ is a phase function in $\R^n$ satisfying Bourgain's condition, then $\phi$-Kakeya sets have Hausdorff dimension $n$. 
\end{conjecture}

\begin{remark}
    One hopes that Conjecture \ref{conj: Bourgain Kakeya} holds for a stronger version of dimension known as a maximal function estimate (see \cite[Section 5]{DemeterBook}), and possibly for further refined incidence questions analogous to those for straight tubes in \cite{wangwurestriction}. We are hopeful that such estimates would lead to direct progress on Conjecture \ref{conj: Hormander dichotomy}, though there might still be other technical challenges.
    If $\phi$ fails Bourgain's condition, the corresponding maximal function version of the $\phi$-Kakeya conjecture is false, as the proof of the negative result in \cite{HormanderDichotomy} shows, but it is unclear whether the conjecture should hold for Hausdorff dimension, of course barring certain worst-case examples (see \cite{GuthSharpOscillatory} for a discussion of such examples). 
    We do not pursue these questions in this work. 
\end{remark}

So far, the best partial progress toward Conjecture \ref{conj: Bourgain Kakeya} is dimension $\geq 2 + \frac{1}{3}$ for $n = 3$ and dimension $\geq (2-\sqrt{2})n + O(1)$ for general $n \geq 3$. These are both a consequence of the result that $\phi$-tubes satisfy strong polynomial Wolff axioms when $\phi$ satisfies Bourgain's condition, and a polynomial partitioning argument from \cite{HickmanRogersZhang} (see also the discussion in \cite[Section 2]{DaiOscillatory}). These estimates match the state of the art for the classical Kakeya conjecture in high dimensions (by either \cite{KatzTaohighdimkakeya} or \cite{HickmanRogersZhang}, which both give the numerology $(2-\sqrt{2})n + O(1)$). In low dimensions, however, the understanding of Conjecture \ref{conj: Bourgain Kakeya} falls short of the classical Kakeya conjecture. 
Notably in dimension $3$ we do not even have an analogue of the classical bound of $5/2$ due to Wolff \cite{wolffHairbrush}. For more discussion of adapting Wolff's hairbrush argument to curves in a slightly different context, see \cite{nadjimzadah2025newcurvedkakeyaestimates}.

\subsection{Geometric characterization of Bourgain's condition}

We find a new purely geometric characterization of Bourgain's condition which is transparently invariant with respect to diffeomorphisms in $\x$ and $\xi$ separately. The characterization quickly leads to 
a reduction from a sticky version of Conjecture \ref{conj: Bourgain Kakeya} in dimension $n$ to the classical sticky Kakeya conjecture in dimension $n$ (the informal version of this is Theorem \ref{thm: sticky reduction informal} and the formal version is Theorem \ref{thm: sticky reduction formal}, and a version in $\R^3$ related to polynomial Wolff axioms is Theorem \ref{thm: curved sticky PWA}).

Stated informally, $\phi$ satisfies Bourgain's condition if and only if for each $(\delta^{1/2},\phi)$-tube $T_0$, the family of $(\delta,\phi)$-tubes contained in $T_0$ is diffeomorphic to a family of straight $\delta$-tubes in a straight $\delta^{1/2}$-tube, while preserving the directions up to diffeomorphism. Thus we found an explicit connection between standard Kakeya sets and $\phi$-Kakeya sets when $\phi$ satisfies Bourgain's condition. Given a curve $\gamma \subset \R^n$, we write $\gamma + O(r)$ to denote the $O(r)$ neighborhood of $\gamma$ below. 
\begin{theorem}[Geometric characterization of Bourgain's condition]\label{thm: geometric characterization of Bourgain}
    A H\"ormander-type phase function $\phi$ satisfies Bourgain's condition if and only if there exist 
    smooth maps $F_{\xi_0,v_0} : M \to \R^n$, $\Xi_{\xi_0,v_0} : \Sigma \to \R^{n-1}$, and $V_{\xi_0,v_0} : \Sigma \times \V \to \R^{n-1}$ depending smoothly on $(\xi_0,v_0) \in \Sigma \times \V$, satisfying the following: 
    \begin{itemize}
        \item The map $F_{\xi_0,v_0} : M \to \R^n$ is a local diffeomorphism.
        \item The map $(\Xi_{\xi_0,v_0}, V_{\xi_0,v_0}) : \Sigma \times \V \to \R^{n-1} \times \R^{n-1}$ is a local diffeomorphism (equivalently, $\Xi_{\xi_0,v_0}$ is a local diffeomorphism and $\nabla_v V_{\xi_0,v_0}$ is invertible). 
        \item Fix $(\xi_0,v_0) \in \Sigma \times \V$ and let $F = F_{\xi_0, v_0}$, $\Xi = \Xi_{\xi_0,v_0}$, and $V = V_{\xi_0, v_0}$. Then 
        \begin{align}\label{eq: bourgain characterization}
            F(\ell_{\xi,v}) \subset \mathrm{line}_{\Xi(\xi), V(\xi,v)} + O(|(\xi,v) - (\xi_0,v_0)|^2)
        \end{align}
        on $\Sigma \times \V$. In particular, if $|(\xi,v)-(\xi_0,v_0)| \leq \delta^{1/2}$, then $F(\ell_{\xi,v} + O(\delta))$ is comparable to $\mathrm{line}_{\Xi(\xi), V(\xi,v)} + O(\delta)$.
    \end{itemize}
\end{theorem}

\begin{remark}\label{remark: data normalizations}
    By harmless transformations of $F_{\xi_0,v_0},\Xi_{\xi_0,v_0},V_{\xi_0,v_0}$, we may assume that 
    \begin{align}
        \Xi_{\xi_0,v_0}(\xi_0) = V_{\xi_0,v_0}(\xi_0,v_0) = 0.
    \end{align}
    This will be seen in the proof of Theorem \ref{thm: geometric characterization of Bourgain} and will be used later. 
\end{remark}

\begin{remark}\label{remark: invertibility follows}
    Once $F_{\xi_0,v_0}$ and $\Xi_{\xi_0,v_0}$ are local diffeomorphisms and \eqref{eq: bourgain characterization} holds, it follows automatically that $\nabla_v V_{\xi_0,v_0}$ is invertible. To see this, we will not use the hypothesis that $\nabla_v V_{\xi_0,v_0}$ is invertible in the proof of the reverse direction of Theorem \ref{thm: geometric characterization of Bourgain}, and we will instead derive it from the other hypotheses. 
\end{remark}

\begin{figure}[!h]
    \centering
    \includegraphics[width=.95\linewidth]{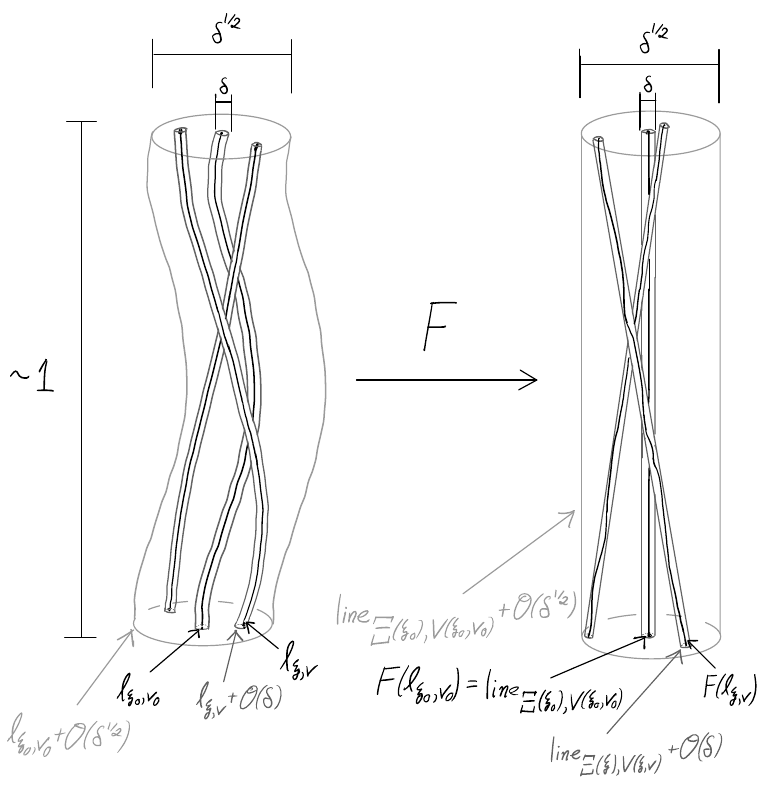}
    \caption{Cartoon of Theorem \ref{thm: geometric characterization of Bourgain}.}
    \label{fig: geometric characterization}
\end{figure}

Figure \ref{fig: geometric characterization} is a cartoon of Theorem \ref{thm: geometric characterization of Bourgain}. 
Theorem \ref{thm: geometric characterization of Bourgain} also has the following consequence. If $\phi$ satisfies Bourgain's condition, then the tubular neighborhood of any $\phi$-curve $\ell_{\xi,v}$ can be put into a normal form via $F = F_{\xi,v}$, in which the $2(n-1)$-parameter family of nearby $\phi$-curves can be simultaneously approximately straightened. For a general phase function $\phi$ that does not satisfy Bourgain's condition, one cannot expect that more than an $(n-1)$-parameter subfamily of $\phi$-curves can be simultaneously straightened. For example, straightening the $(n-1)$-parameter subfamily of $\phi$-curves in a fixed direction $\xi_0 \in \Sigma$ uses up essentially all the degrees of freedom of a diffeomorphism in $\x$, and so does straightening the $(n-1)$-parameter subfamily of $\phi$-curves passing through a point $\x_0 \in M$.

We can see the diffeomorphism invariance in $\x$ and $\xi$ of our new characterization of Bourgain's condition quickly. Let $\phi$ be a phase function satisfying Bourgain's condition with corresponding $\phi$-curves $\ell_{\xi,v}$, and let $G : M' \to M$ and $H : \Sigma' \to \Sigma$ be diffeomorphisms. Define the transformed phase function $\tilde \phi : M' \times \Sigma' \to \R$ by $\tilde \phi(\x,\xi) = \phi(G(\x), H(\xi))$. Then the $\tilde \phi$-curves $\tilde \ell_{\xi,v}$ transform like 
\begin{align}
    \tilde \ell_{\xi,v} &= G^{-1}(\ell_{H(\xi), v \cdot\nabla H(\xi)^{-1}}).
\end{align}
Thus the geometric characterization holds with $F_{\xi_0,v_0}$ replaced by $F_{\xi_0,v_0} \circ G$, $\Xi_{\xi_0,v_0}$ replaced by $\Xi_{\xi_0,v_0} \circ H$, and $V_{\xi_0,v_0}$ replaced by $(\xi,v) \mapsto V_{\xi_0,v_0}(H(\xi),v \cdot \nabla H(\xi)^{-1})$. We will introduce examples which show that Theorem \ref{thm: geometric characterization of Bourgain} cannot be strengthened to have an error $O(|(\xi,v) - (\xi_0,v_0)|^4)$.

\begin{remark}
    The diffeomorphisms $\Xi_{\xi_0,v_0}$ in $\xi$ are essential in Theorem \ref{thm: geometric characterization of Bourgain}. Bourgain found the example 
    \begin{align}
        \phi_{\mathrm{worst}}(\x,\xi) = x_1\xi_1 + x_2\xi_2 + t\xi_1 \xi_2 + \tfrac{1}{2} t^2 \xi_2^2, \quad (0_M, 0_\Sigma) =(0,0)\in \R^{3} \times \R^2,
    \end{align}
    which satisfies H\"ormander's conditions, fails Bourgain's condition, and in fact gives the worst case bounds for H\"ormander's problem \cite{BourgainSeveralVariables} (see \cite{GuthSharpOscillatory, worstCaseArbSig} for a study of this phenomenon). The $\phi_{\mathrm{worst}}$-curves are 
    \begin{align}
        \ell_{\xi,v} = \{(v_1 - t\xi_2, v_2 - t\xi_1 - t^2 \xi_2,t) : |t| \leq 1\}.
    \end{align}
    The diffeomorphism $F(\x) = (x_1,x_2 - tx_1, t)$ takes the curves $\ell_{\xi,v}$ to lines,
    \begin{align}
        F(\ell_{\xi,v}) \subset \mathrm{line}_{(\xi_2,\xi_1+v_1), v},
    \end{align}
    but crucially $(\xi,v) \mapsto (\xi_2,\xi_1 + v_1)$ is \emph{not} a diffeomorphism in $\xi$ alone. 
\end{remark}

The following mild reformulation of Bourgain's condition is important in proving Theorem \ref{thm: geometric characterization of Bourgain}. It will also be crucial in constructing Example \ref{ex: tan example} which we will introduce shortly. 
\begin{proposition}[Mild reformulation of Bourgain's condition]\label{prop: mild reformulation Bourgain}
    Let $\phi : M \times \Sigma \to \R$ be a H\"ormander-type phase function satisfying Bourgain's condition. Then there exist a smooth scalar valued function $c : M \times \Sigma \to \R$ and two smooth symmetric-matrix-valued functions $A,B : \V \times \Sigma \to \mathrm{Sym}_{n-1}(\R)$ such that 
    \begin{align}\label{eq: ABc equation}
        \nabla_\xi^2 \phi(\x,\xi) = A(\nabla_\xi \phi(\x,\xi), \xi) + c(\x,\xi)B(\nabla_\xi \phi(\x,\xi), \xi)
   \end{align}
   on $M \times \Sigma$. Additionally $B(v,\xi)$ is a non-degenerate matrix for each $(v,\xi) \in \V \times \Sigma$ and $(G(\x,\xi) \cdot \nabla_\x)c(\x,\xi) \neq 0$ on $M \times \Sigma$.  
   
   Conversely, suppose that $\phi : M \times \Sigma \to \R$ is a smooth function satisfying (H1). Suppose there are smooth $c : M \times \Sigma \to \R$, $A : \V \times \Sigma \to \R^{(n-1)\times(n-1)}$, and $B : \V \times \Sigma \to \R^{(n-1)\times(n-1)}$ such that \eqref{eq: ABc equation} holds,
   \begin{itemize}
       \item $B(v,\xi)$ is nondegenerate (resp. positive-definite), and 
       \item $(G(\x,\xi) \cdot \nabla_\x) c(\x,\xi) \neq 0$.
   \end{itemize}
   Then $\phi$ satisfies (H2) (resp. (H2+)) and Bourgain's condition.
\end{proposition}

\begin{example}[Proposition \ref{prop: mild reformulation Bourgain} for the restriction and Bochner--Riesz examples]
    The phase functions $\phi_{n,\rest}$ and $\phi_{n,\BR}$ can be written in the form \eqref{eq: ABc equation} by making the following choices of $A,B,c$:
    \begin{itemize}
        \item $\phi_{n,\rest}$: 
        \begin{align}
            A(v,\xi) &= 0, \\
            B(v,\xi) &= I_{n-1}, \\
            c(\x,\xi) &= t.
        \end{align}
        \item $\phi_{n,\BR}$:
        \begin{align}
            A(v,\xi) &= 0, \\
            B_{ij}(v,\xi) &= \begin{cases}
                1-v_i^2, & i=j \\
                -v_i v_j, & i \neq j
            \end{cases},\\
            c(\x,\xi) &= \frac{t}{\sqrt{1 + |x - t\xi|^2}}.
        \end{align}
    \end{itemize}
\end{example}

\subsection{Sticky $\phi$-Kakeya sets, and a reduction to classical sticky Kakeya sets under Bourgain's condition}

It was demonstrated in the series of landmark papers of Wang--Zahl \cite{wangZahlstickyKakeya, WangZahlAssouadKakeya, wangZahlKakeya} that the sticky Kakeya conjecture is true in $\R^3$, and that the Kakeya conjecture in $\R^3$ can be reduced to this sticky case. 
It is expected that the sticky Kakeya conjecture is true in higher dimensions, and that a general to sticky reduction is plausible in higher dimensions (see the discussion in \cite[Section 9]{guth2025introductionproofkakeyaconjecture}). 
Roughly speaking, a Kakeya set is sticky if it has a self-similar structure. 
This idea can be captured in different ways. In \cite{wangZahlstickyKakeya}, a Kakeya set $K \subset \R^n$ is said to be sticky if its set of (direction separated) lines has packing dimension $n-1$. 
However the general to sticky reduction in \cite{wangZahlKakeya} requires one to study sets of tubes which are not necessarily direction separated, but instead only satisfy the more general Katz-Tao or Frostman convex Wolff axioms. 
Correspondingly, a version of the sticky Kakeya conjecture in $\R^3$ involving these conditions is needed (see \cite[Theorem 7.3]{guth2026streamlinedproofkakeyaset}), building off \cite{wangZahlstickyKakeya,WangZahlAssouadKakeya}. 

It seems plausible that the sticky case is also the enemy case for Conjecture \ref{conj: Bourgain Kakeya}. 
We formalize stickiness for $\phi$-Kakeya sets by generalizing the definition in \cite{wangZahlstickyKakeya}, and then give the reduction from classical sticky Kakeya to curved sticky Kakeya for Bourgain's condition. 
We can also show a sticky $\phi$-Kakeya theorem corresponding to Frostman \emph{polynomial} Wolff axioms in $\R^3$. We discuss this in Section \ref{subsec: future work} and give the rigorous argument in Appendix \ref{subsec: curved sticky PWA}. 

To measure how closely $\phi$-curves pack together, we equip $\mathcal C(\phi)$ with the metric 
\begin{align}
    d(\ell_{\xi,v},\ell_{\xi',v'}) = |\xi-\xi'| + |v - v'|.
\end{align}
In the special case of $\phi = \phi_{n,\rest}$, this agrees with the metric on lines used in \cite{wangZahlstickyKakeya} up to bi-Lipschitz equivalence. 
We are led to the following generalization of classical sticky Kakeya sets from \cite{wangZahlstickyKakeya} to sticky $\phi$-Kakeya sets.

\begin{definition}[Sticky $\phi$-Kakeya set]
    A compact set $K \subset M_0$ is called a sticky $\phi$-Kakeya set if there is a set of $\phi$-curves $L \subset \mathcal C(\phi)$ with packing dimension $n-1$ that contains at least one curve for each direction $\xi \in \Sigma_0$, so that $\ell \cap M_0 \subset K$ for each $\ell \in L$. 
    More generally, a $\phi$-Kakeya set is $\eta$-close to being sticky if the packing dimension of $L$ is at most $n-1 + \eta$. 
\end{definition}

A classical sticky Kakeya set is then the same as a sticky $\phi_{n, \rest}$-Kakeya set. 
We are led to the sticky case of Conjecture \ref{conj: Bourgain Kakeya}.
\begin{conjecture}[Curved sticky Kakeya conjecture for Bourgain's condition]\label{conj: sticky Bourgain kakeya}
		If $\phi$ is a phase function in $\R^n$ satisfying Bourgain's condition, then sticky $\phi$-Kakeya sets have Hausdorff dimension $n$.
\end{conjecture}
By using Theorem \ref{thm: geometric characterization of Bourgain} and a simple induction on scales argument, we reduce Conjecture \ref{conj: sticky Bourgain kakeya} in dimension $n$ to the classical sticky Kakeya conjecture in dimension $n$. 
\begin{theorem}[(Informal) Classical sticky Kakeya implies curved sticky Kakeya for Bourgain's condition]\label{thm: sticky reduction informal}
		If classical sticky Kakeya sets in $\R^n$ have dimension $n$, then for all phase functions $\phi$ in $\R^n$ satisfying Bourgain's condition, sticky $\phi$-Kakeya sets have dimension $n$. 
        In particular, Conjecture \ref{conj: sticky Bourgain kakeya} is true when $n=3$. 
\end{theorem}

To make the implication in Theorem \ref{thm: sticky reduction informal} formal, we need to introduce a discretization scheme and notation analogous to \cite{wangZahlstickyKakeya}. 

We will always assume going forward that a $(\delta,\phi)$-tube $T = T_{\xi,v}^\delta$ has $(\xi,v) \in \Sigma_0 \times \V_0$. 
A \emph{shading} of $T$ is a measurable subset $Y(T) \subset T \cap M_0$. Write $(\T,Y)_\delta$ for a collection of $(\delta,\phi)$-tubes and their associated shadings. 
We say two $(\delta,\phi)$-tubes $T_1=T^\delta_{\xi_1,v_1}$ and $T_2=T^\delta_{\xi_2,v_2}$ are \emph{essentially distinct} if $d(\ell_{\xi_1,v_1}, \ell_{\xi_2,v_2}) \geq \delta$. We say $T_1$ and $T_2$ are \emph{essentially parallel} if $|\dir(\ell_{\xi_1,v_1})-\dir(\ell_{\xi_2,v_2})| \leq \delta$ (we will omit \emph{essentially} for brevity when the context is clear).
Let $T$ be a $(\delta,\phi)$-tube and $\tilde T$ a $(\rho, \phi)$-tube with $\delta \leq \rho$. We say $\tilde T$ \emph{covers} $T$ if their coaxial curves $\ell, \tilde \ell$ satisfy $d(\ell, \tilde \ell) \leq c_{\mathrm{cover}} \rho$. We choose $c_{\mathrm{cover}} > 0$ here to be small enough that when $\delta \leq \rho/2$ and $\tilde T$ covers $T$, then $T \cap M_0 \subset \tilde T$ (such a constant is guaranteed to exist by the smoothness of $\phi$ and the compactness of $M_0, \Sigma_0, \V_0$).
If $\T$ is a collection of $(\delta,\phi)$-tubes and $\tilde \T$ is a collection of $(\rho,\phi)$-tubes, we say $\tilde \T$ covers $\T$ if each tube in $\T$ is covered by some tube in $\tilde \T$. Write $\T[\tilde T]$ for the set of tubes in $\T$ that are covered by $\tilde T$.

We may now define the discretized version of sticky $\phi$-Kakeya sets, and what it means to have dimension $n$. 
\begin{definition}[Assertion $\SK(\phi, \eps, \eta, \delta_0)$]
\label{def: SK assertion}
    Fix a phase function $\phi$. We say that $\SK(\phi, \eps, \eta, \delta_0)$ holds if for all $\delta \in (0,\delta_0]$ and all pairs $(\T, Y)_\delta$ satisfying the properties
    \begin{enumerate}
        \item[(a)] the $(\delta,\phi)$-tubes in $\T$ are essentially distinct,
        \item[(b)] for each $\delta \leq \rho \leq 1$, $\T$ can be covered by a set of $(\rho,\phi)$-tubes, at most $\delta^{-\eta}$ of which are essentially parallel to a common tube, and
        \item[(c)] $\sum_{T \in \T} |Y(T)| \geq \delta^\eta$,
    \end{enumerate}
    one has 
    \begin{align}
        |\bigcup_{T \in \T} Y(T)| \geq \delta^\eps. 
    \end{align}
\end{definition}
The discretized sticky $\phi$-Kakeya conjecture may now be stated as follows.
\begin{conjecture}[Discretized sticky $\phi$-Kakeya conjecture]\label{conj: sticky phi kakeya discretized}
    For every $\eps > 0$ there exists $\eta = \eta_\phi(\eps) > 0$ and $\delta_{0} = \delta_{0,\phi}(\eps) > 0$ such that $\mathrm{SK}(\phi, \eps, \eta, \delta_0)$ holds. 
\end{conjecture}
The discretized classical sticky Kakeya conjecture in $\R^n$ is then Conjecture \ref{conj: sticky phi kakeya discretized} applied to $\phi_{n,\mathrm{rest}}$.
\begin{conjecture}[Discretized classical sticky Kakeya conjecture in $\R^n$]\label{conj: discr classical sticky Kakeya conj}
    For every $\eps > 0$ there exists $\eta = \eta_{n,\rest}(\eps) > 0$ and $\delta_0 = \delta_{0,n,\rest}(\eps) > 0$ such that $\SK(\phi_{n,\rest},\eps,\eta,\delta_0)$ holds. 
\end{conjecture}
The main result of \cite{wangZahlstickyKakeya} can then be formulated as the following discretized statement. 
\begin{theorem}[Discretized classical sticky Kakeya in $\R^3$ \cite{wangZahlstickyKakeya}]\label{thm: wang zahl sticky}
    Conjecture \ref{conj: discr classical sticky Kakeya conj} is true in the case $n =3$. 
\end{theorem}

Analogously to \cite{wangZahlstickyKakeya}, Conjecture \ref{conj: sticky phi kakeya discretized} implies the following Hausdorff dimension statement. Below $\dim_\H$ denotes Hausdorff dimension.
\begin{proposition}[Discretized to dimension statement] \label{prop: discretized implies dimension}
    Suppose that Conjecture \ref{conj: sticky phi kakeya discretized} holds for a phase function $\phi$ in $\R^n$. Then for all $\eps > 0$, there is $\eta > 0$ so that the following holds. Let $K$ be a $\phi$-Kakeya set that is $\eta$-close to being sticky. Then $\dim_\H K \geq n - \eps$. In particular, if $K$ is a sticky $\phi$-Kakeya set then $\dim_\H K = n$. 
\end{proposition}

We now state the formal version of Theorem \ref{thm: sticky reduction informal}.
\begin{theorem}[(Formal) Classical sticky Kakeya implies curved sticky Kakeya for Bourgain's condition]\label{thm: sticky reduction formal}
\leavevmode
    \begin{enumerate}
        \item[(i)] Let $n \geq 3$ and suppose that $\phi$ is a phase function in $\R^n$ satisfying Bourgain's condition. 
        If Conjecture \ref{conj: discr classical sticky Kakeya conj} holds in $\R^n$, then Conjecture \ref{conj: sticky phi kakeya discretized} holds for $\phi$. 
        \item[(ii)] Conjecture \ref{conj: sticky phi kakeya discretized} is true for $\phi$ satisfying Bourgain's condition in $\R^3$ (by using (i) and Theorem \ref{thm: wang zahl sticky}), 
        \item[(iii)] Conjecture \ref{conj: sticky Bourgain kakeya} is true in $\R^3$ (by using (ii) and Proposition \ref{prop: discretized implies dimension}).
    \end{enumerate}
\end{theorem}

Below we describe some sharpness examples for Theorem \ref{thm: geometric characterization of Bourgain}, which also demonstrate that a general to sticky reduction for Conjecture \ref{conj: Bourgain Kakeya} in the spirit of Wang--Zahl requires new ideas. We discuss steps in the direction of a general to sticky reduction in Section \ref{subsec: future work}.

\subsection{Sharpness examples for Theorem \ref{thm: geometric characterization of Bourgain} and barriers to a general to sticky reduction}

For each $n \geq 3,$ we find a phase function $\phi$ in $\R^n$ satisfying Bourgain's condition, for which there does not exist a diffeomorphism taking the family of $\phi$-curves to lines.
This answers in the negative a question raised by Gao--Liu--Xi \cite{gao2025curvedkakeyasetsnikodym}.

\begin{example}[The $\tan$-example in $\R^n$]\label{ex: tan example}
    Fix $n \geq 3$ and write $\xi = (\xi', \xi_{n-1}) \in \R^{n-2} \times \R$, $\x = (x',x_{n-1},t) \in \R^{n-2} \times \R \times \R$. The $\tan$-example in $\R^n$ is the phase function 
    \begin{align}
        \phi_{n,\tan}(x,\xi) = x' \cdot \xi' + \tfrac{1}{2} t^2 |\xi'|^2 + \log(\sec(t \xi_{n-1} + x_{n-1})), \quad (0_M,0_\Sigma) = ((0,0,1),(0,0)).
    \end{align}
	This example is a positive-definite phase function satisfying Bourgain's condition. The corresponding $\phi_{n,\tan}$-curves are 
    \begin{align}
        \ell_{\xi,v} = \{(v' - t^2 \xi', \tan^{-1}(\frac{v_{n-1}}{t}) - t\xi_{n-1}, t) : |t-1| \leq 1/10\}.
    \end{align}
\end{example}
We call $\phi_{n,\tan}$ the $\tan$-example since the family of curves involves the function $\tan$. 
In Section \ref{sec: tan example}, we will construct this example using Proposition \ref{prop: mild reformulation Bourgain}. We then prove that there is no diffeomorphism from the $(\delta, \phi_{n,\tan})$-tubes in a fixed $(\delta^{1/4}, \phi_{n,\tan})$-tube to straight $\delta$-tubes in a straight $\delta^{1/4}$-tube. The $\tan$-example appears to be the simplest to write example satisfying Bourgain's condition with this property. 

\begin{proposition}[$\tan$-example is curved]\label{prop: tan example curved}
        For any fixed $\eps_0 > 0$, there does not exist a diffeomorphism $F : B_{\eps_0}(0_M) \to   \R^n$ onto its image and smooth $\Xi,V : \R^{n-1} \times \R^{n-1} \to \R^{n-1}$ such that 
        \begin{align}\label{eq: lines to error 4}
            F(\ell_{\xi,v} \cap B_{\eps_0}(0_M)) \subset \mathrm{line}_{\Xi(\xi,v), V(\xi,v)}+ O(|(\xi,v)|^4)
        \end{align}
        for $(\xi,v)$ near $(0,0)$. 
        In particular, the family of $\phi_{n,\tan}$-curves is not diffeomorphic to a family of lines. 
\end{proposition}
We will prove Proposition \ref{prop: tan example curved} by showing that the union of curves through $\ell_{0,0}$ and a point $\mathbf{p} \notin \ell_{0,0}$ is not contained in a surface, while of course a union of lines passing through a fixed line and point is contained in a plane. In \cite{nadjimzadah2025newcurvedkakeyaestimates} this type of behavior is known as \emph{coniness}, where it was exploited to give positive results for Kakeya sets of curves in a slightly different context. Figure \ref{fig: tan failure diagram} is a cartoon of this failure of \eqref{eq: lines to error 4} for $\phi_{n,\tan}$.

\begin{figure}[!h]
    \centering
    \includegraphics[width=.66\linewidth]{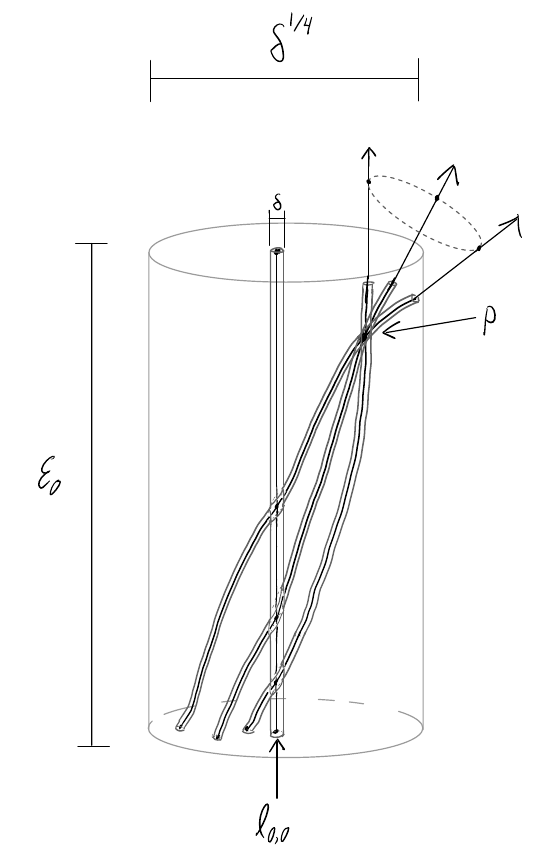}
    \caption{Cartoon of Proposition \ref{prop: tan example curved}.}
    \label{fig: tan failure diagram}
\end{figure}

Proposition \ref{prop: tan example curved} shows that the error term in Theorem \ref{thm: geometric characterization of Bourgain} cannot be strengthened to $O(|(\xi,v)-(\xi_0,v_0)|^4)$, and we suspect that a version with error $O(|(\xi,v) - (\xi_0,v_0)|^3)$ is also impossible. One should contrast Example \ref{ex: tan example} with translation-invariant phase functions and phase functions arising from Carleson--Sj\"olin operators on manifolds. If the phase function $\phi$ in either of those cases satisfies Bourgain's condition, then $\mathcal C(\phi)$ is diffeomorphic to a family of lines \cite{DaiOscillatory, gao2025curvedkakeyasetsnikodym}.
Not only is $\mathcal C(\phi_{n,\tan})$ not diffeomorphic to a family of lines, which we believe is the first example satisfying Bourgain's condition with this property, 
it is not even diffeomorphic to a family of lines after allowing $O(|(\xi,v)|^4)$ perturbations.

Numerous arguments in the study of Kakeya sets ultimately use the property that lines may be efficiently organized into planes. For example, Wolff's hairbrush argument uses that the $\delta$-tubes through a fixed stem can be organized into essentially disjoint slabs \cite{wolffHairbrush}. The argument of Wang--Zahl uses that $\delta$-tubes can be organized into planks \cite{wangZahlKakeya}. At least outside a fixed $(\delta^{1/4}, \phi_{n,\tan})$-tube, the $(\delta, \phi_{n, \tan})$-tubes do not have properties analogous to these. Thus to go beyond the sticky case of Conjecture \ref{conj: Bourgain Kakeya}, new ideas are needed.

\subsection{Steps toward a general to sticky reduction for curved Kakeya}\label{subsec: future work}
    A promising approach to Conjecture \ref{conj: Bourgain Kakeya} in $\R^3$ is to carry out a general to sticky reduction in the spirit of Wang--Zahl,
    using semi-algebraic sets instead of convex sets. We record several initial steps toward such a reduction and briefly indicate some of the difficulties that arise. The supporting results are proved in Appendix \ref{sec: appendix}.

    Recall that semi-algebraic sets are generated by finite unions, intersections, and complements of $\{x : P(x) = 0\}$ and $\{x : Q(x) > 0\}$ over all polynomials $P$ and $Q$ over $\R^n$, and the complexity of a semi-algebraic set is the smallest sum of degrees of the polynomials appearing in a description of the set.

    We will replace $\phi$-tubes by semi-algebraic approximations inside a ball, and reduce Conjecture \ref{conj: Bourgain Kakeya} to a natural conjecture about semi-algebraic tubes satisfying polynomial Wolff axioms. 
    To this end, take $\delta \in (0,1]$, $\beta \in (0,1)$, and fix a ball $B$ of radius $\delta^\beta$ (one should think of $\beta$ very small). Let $\rho \in [\delta, \delta^\beta)$ and consider a $(\rho, \phi)$-tube $T$ passing through $(1/2)B$. We can define a length $\sim \delta^\beta$ and radius $\sim \rho$ curved tube segment $T^B$ which is semi-algebraic of complexity $\leq 100 \beta^{-1}$, and satisfies $T \cap B \subset T^B$. We call these $(\rho,\phi,B)$-tubelets, and we construct them in Appendix \ref{subsec: curved sticky PWA}. A $\lambda$-dense shading of $T^B$ is $Y^B(T^B) \subset T^B$ with $|Y^B(T^B)| \geq \lambda|T^B|$. 
    Since $(\rho,\phi,B)$-tubelets have length $\delta^\beta$, 
    a set $\T^B$ of such tubelets is naturally direction separated if the $(\rho,\phi)$-tubes they come from are $\rho/\delta^\beta$-direction separated. 

    Analogously to \cite{guth2026streamlinedproofkakeyaset}, if $\T^B$ is a set of tubelets and $S$ is a semi-algebraic set, we define $\T^B[S] = \{T^B \in \T^B : T^B \subset S\}$, $\Delta(\T^B,S) = |\T^B[S]| |T^B|/|S|$, and the maximal density 
    \begin{align}
        \Delta_{\max}^E (\T^B) = \sup_{S \subset B} \Delta(\T^B,S),
    \end{align}
    where $E \geq 10^{5} \beta^{-1}$ and the supremum is taken over $S \subset B$ semi-algebraic of complexity $\leq E$. 
    The analogue of the Kakeya conjecture with convex sets, \cite[Theorem 1.1]{guth2026streamlinedproofkakeyaset}, is then as follows. One should compare this with the classical Kakeya conjecture for polynomial Wolff axioms in $\R^n$, \cite[Conjecture 1.1]{guthZahlR4}, as well. 
    
    \begin{conjecture}[Curved Kakeya conjecture for polynomial Wolff axioms in $\R^3$]\label{conj: PWA Kak conj}
        Suppose $\phi$ is a phase function in $\R^3$ satisfying Bourgain's condition. Fix $\beta \in (0,1)$. There exists $E = E(\beta) \geq 10^5 \beta^{-1}$ such that the following holds. 

        For all $\eps > 0$, there is $\eta=\eta_{\PW}(\eps, E) > 0$ and $\delta_{0,\PW}(\eps,E) > 0$ so that the following holds for all $\delta \in (0,\delta_0]$. If $B$ is a $\delta^\beta$-ball and $(\T^B,Y^B)$ is a set of $(\delta,\phi,B)$-tubelets with $\Delta_{\max}^E(\T^B) \leq \delta^{-\eta}$ and $\delta^\eta$-dense shadings, then 
        \begin{align}
            |\bigcup_{T^B \in \T^B} Y^B(T^B)| \geq \delta^\eps |\T^B| |T^B|.
        \end{align}
    \end{conjecture}

    \begin{remark}
        By Theorem \ref{thm: dir sep implies PWA}, a $\delta/\delta^\beta$-direction separated family of $(\delta, \phi, B)$-tubelets $\T^B$ satisfies $\Delta_{\max}^E(\T^B) \lesssim_{E,\eta} \delta^{-\eta}$.
        By linearization and rescaling, the case $\beta \in [1/2,1)$ is already implied by the Kakeya theorem \cite[Theorem 1.1]{guth2026streamlinedproofkakeyaset} (even if $\phi$ is only assumed to be H\"ormander-type). 
    \end{remark}

    By pigeonholing and using Theorem \ref{thm: dir sep implies PWA}, Conjecture \ref{conj: PWA Kak conj} implies Conjecture \ref{conj: Bourgain Kakeya} in $\R^3$. We leave the proof to Appendix \ref{subsec: PWA conj implies dir sep conj}.

    \begin{restatable}[Conjecture \ref{conj: PWA Kak conj} implies Conjecture \ref{conj: Bourgain Kakeya} in $\R^3$]{proposition}{propPWAconjimpliesdirsepconj}\label{prop: PWA conj implies dir sep conj}
        Suppose that Conjecture \ref{conj: PWA Kak conj} holds, and let $\phi$ be a phase function in $\R^3$ satisfying Bourgain's condition. 

        Then for all $\eps > 0$, there exists $\eta > 0$ and $\delta_0 > 0$ such that the following holds for all $\delta \in (0,\delta_0]$. For all pairs $(\T,Y)$ of $\delta$-direction separated $(\delta,\phi)$-tubes with $\delta^\eta$-dense shadings, 
        \begin{align}
            |\bigcup_{T \in \T} Y(T)| \geq \delta^\eps |\T| |T|. 
        \end{align}
        By standard discretization arguments (see e.g. \cite{wolffNotes}), this implies Conjecture \ref{conj: Bourgain Kakeya} in $\R^3$.
    \end{restatable}

    There are substantial new complications when one passes to semi-algebraic sets and curved tubes. In the case of classical Kakeya, an estimate for a union of convex sets can essentially be reduced to an estimate for a union of tubes (see \cite[Lemma 6.1, 6.4]{guth2026streamlinedproofkakeyaset}). For the $\tan$-example, for instance, one is led to study the incidence geometry of certain \emph{twisted} prisms, which does not easily reduce to an incidence problem for curved tubes. We plan to investigate these new complications in future work.  

    We do, however, obtain some analogues of important ingredients in \cite{guth2026streamlinedproofkakeyaset}. We are able to prove a curved Kakeya estimate under a multiscale Frostman polynomial Wolff assumption on the tubes---Theorem \ref{thm: curved sticky PWA} below. This is the analogue of the Kakeya estimate under a multiscale Frostman convex Wolff assumption on the tubes, \cite[Theorem 7.3(A)]{guth2026streamlinedproofkakeyaset}, and we expect it to be an important input to proving Conjecture \ref{conj: PWA Kak conj}. To state Theorem \ref{thm: curved sticky PWA}, we need a few more definitions. The Frostman polynomial Wolff constant is 
    \begin{align}
        \mathrm{PW}^E_F(\T^B,S) = \frac{\sup_{S' \subset S} \Delta(\T^B,S')}{\Delta(\T^B,S)}, 
    \end{align}
    where the supremum is taken over $S' \subset S$ semi-algebraic of complexity $\leq E$. We can define a set $\T^B$ of $(\rho, \phi, B)$-tubelets (and shadings) to be uniform just as in \cite[Definition 2.1]{guth2025outlinewangzahlproofkakeya} (and \cite[Definition 2.2]{guth2025outlinewangzahlproofkakeya}). We call $\T^B$ $(C,E)$-\emph{Frostman polynomial Wolff at every scale} if it is uniform and for every $\rho \leq \rho' \leq \delta^\beta$, we have 
    \begin{align}
        \PW_F^E(\T^B[T_{\rho'}^B],T_{\rho'}^B) \leq C
    \end{align}
    for every $T_{\rho'}^B \in \T_{\rho'}^B$ (this is the analogue of \cite[Definition 7.1(A)]{guth2026streamlinedproofkakeyaset}).

    \begin{restatable}[Curved version of \protect{\cite[Theorem 7.3(A)]{guth2026streamlinedproofkakeyaset}}]{theorem}{thmcurvedstickyPWA}\label{thm: curved sticky PWA}
        Assume $\phi$ is a phase function in $\R^3$ satisfying Bourgain's condition. Fix $\beta \in (0,1)$. There exists $E \geq 10^5 \beta^{-1}$ such that the following holds (one can in fact take $E = 10^5 \beta^{-1}$). 
        
        For all $\eps > 0$, there exists $\eta = \eta_{\PW}(\eps,E)$ and $\delta_0 = \delta_{0,\PW}(\eps,E) > 0$ so that the following holds for all $\delta \in (0,\delta_0]$. Let $B$ be a $\delta^\beta$-ball
        and $(\T^B,Y^B)$ a uniform set of $(\delta, \phi, B)$-tubelets with $\delta^{\eta}$-dense shadings. If $\T^B$ is $(\delta^{-\eta}, E)$-Frostman polynomial Wolff at every scale, then 
        \begin{align}
            |\bigcup_{T^B \in \T^B} Y^B(T^B)| \geq \delta^\eps |B|.
        \end{align}
    \end{restatable}

    The proof is similar to Theorem \ref{thm: sticky reduction formal}, but instead we need to show that if a set of $(\rho,\phi,B)$-tubelets is Frostman polynomial Wolff at every scale, then the locally straightened tubelets (using approximations of the maps $F_{\xi_0,v_0}$ from Theorem \ref{thm: geometric characterization of Bourgain}) are Frostman \emph{convex} Wolff at every scale. It uses \cite[Theorem 7.3(A)]{guth2026streamlinedproofkakeyaset} as input. 
    We leave the proof to Appendix \ref{subsec: curved sticky PWA}.

\subsection{Organization of article}
\begin{itemize}    
    \item In Section \ref{sec: mild reformulation of Bourgain} we prove Proposition \ref{prop: mild reformulation Bourgain}. 
    \item In Section \ref{sec: geometric characterization of Bourgain} we prove Theorem \ref{thm: geometric characterization of Bourgain} using Proposition \ref{prop: mild reformulation Bourgain}.
    \item In Section \ref{sec: discretizing} we sketch how to discretize the sticky $\phi$-Kakeya problem in a similar way to \cite[Section 2.3]{wangZahlstickyKakeya}, and sketch the proof of Proposition \ref{prop: discretized implies dimension}.
    \item In Section \ref{sec: sticky reduction} we prove Theorem \ref{thm: sticky reduction formal}.
    \item In Section \ref{sec: tan example} we construct Example \ref{ex: tan example} using Proposition \ref{prop: mild reformulation Bourgain}, and prove Proposition \ref{prop: tan example curved}.
    \item In Appendix \ref{sec: appendix} we prove Theorem \ref{thm: dir sep implies PWA}, Proposition \ref{prop: PWA conj implies dir sep conj}, and 
    Theorem \ref{thm: curved sticky PWA}. 
\end{itemize}

\subsection{Thanks}

The author is grateful to Terence Tao for many helpful conversations about this project, and especially for suggesting the argument for Proposition \ref{prop: mild reformulation Bourgain}. The author also thanks Hong Wang, Shaoming Guo, and Siddharth Mulherkar for their encouragement and helpful conversations. Finally, the author thanks the anonymous referees for their careful reading of an earlier version of this manuscript and for comments that substantially improved its accuracy and presentation.

\section{Proof of Proposition \ref{prop: mild reformulation Bourgain}}\label{sec: mild reformulation of Bourgain}

Recall the Gauss map $G(\x,\xi) = \bigwedge_{j=1}^{n-1} \partial_{\xi_j} \nabla_\x \phi(\x,\xi)$. We will make use of the fact that 
\begin{align}\label{eq: G kills xi phi}
    (G(\x,\xi) \cdot \nabla_\x) \nabla_\xi \phi(\x,\xi) = 0
\end{align}
below. 

\subsection{Forward implication}
\begin{proof}
Assume that $\phi$ satisfies Bourgain's condition and fix $\xi_0 \in \Sigma$.
Since $G(\cdot,\xi_0)$ is a nonzero vector field on $M$, there are local coordinates, depending smoothly on the choice of $\xi_0$, for which $G(\x, \xi_0) = e_n$. 
In the new coordinates \eqref{eq: G kills xi phi} becomes $\partial_{x_n} \nabla_\xi \phi(\x,\xi_0) = 0$.
Thus $\nabla_\xi \phi(\x,\xi_0)$ depends only on $x = (x_1,\ldots, x_{n-1})$, and by (H1) is therefore diffeomorphic to $x$ (depending smoothly on the choice of $\xi_0$).

In the new coordinates \eqref{eq: B cond original} in the statement of Bourgain's condition (Definition \ref{def: Bourgain cond original}) becomes 
\begin{align}\label{eq: B new coords}
    \partial_{x_n}^2 \nabla_\xi^2 \phi(\x,\xi_0) = \lambda(\x,\xi_0)\partial_{x_n} \nabla_\xi^2 \phi(\x,\xi_0).
\end{align}
We may solve this differential equation in $x_n$ to find smooth symmetric-matrix-valued functions $\tilde A, \tilde B : \V \times \Sigma \to \mathrm{Sym}_{n-1}(\R)$, and a smooth function $\tilde c : M \times \Sigma \to \R$ such that 
\begin{align}
    \nabla_\xi^2 \phi(\x, \xi_0) = \tilde A(x,\xi_0) + \tilde c(\x,\xi_0) \tilde B(x,\xi_0). 
\end{align}
By using that $\nabla_\xi \phi(\x,\xi)$ is diffeomorphic to $x$ and then undoing the change of variables in $\x$,  
we obtain $A,B,c$ satisfying
\begin{align}\label{eq: sec 2: B cond}
    \nabla_\xi^2 \phi(\x,\xi_0) = A(\nabla_\xi\phi(\x,\xi_0),\xi_0) + c(\x,\xi_0) B(\nabla_\xi \phi(\x,\xi_0),\xi_0).
\end{align}
Since the change of variables in $\x$ depends smoothly on the choice of $\xi_0$, the functions $A(v,\xi_0),B(v,\xi_0),c(\x,\xi_0)$ are also smooth in $\xi_0$. 

By applying the vector field $G(\x,\xi_0) \cdot \nabla_\x$ to \eqref{eq: sec 2: B cond}, we see from (H2) that $(G(\x, \xi_0) \cdot \nabla_\x)c(\x, \xi_0) \neq 0$ and $B(\nabla_\xi \phi(\x,\xi_0),\xi_0)$ is nondegenerate. When (H2+) is satisfied, $B(\nabla_\xi \phi(\x,\xi_0),\xi_0)$ is furthermore positive-definite. 
\end{proof}

\subsection{Reverse implication}

\begin{proof}
Suppose that we have data $A,B,c$ such that \eqref{eq: ABc equation} holds. 
Applying the operator $(G(\x, \xi) \cdot \nabla_\x)$ to both sides of \eqref{eq: ABc equation} and using \eqref{eq: G kills xi phi}, 
\begin{align}\label{eq: 1 reverse implication mild gen}
	(G(\x,\xi) \cdot \nabla_\x)\nabla_\xi^2 \phi(\x,\xi) &= ((G(\x,\xi) \cdot \nabla_\x) c(\x,\xi)) B(\nabla_\xi \phi(\x,\xi), \xi).
\end{align}
If $(G(\x,\xi) \cdot \nabla_\x)c(\x,\xi) \neq 0$ and $B$ is nondegenerate (resp. positive-definite), then (H2) (resp. (H2+)) holds. 
Applying the operator $G \cdot \nabla_\x$ once again, 
\begin{align}\label{eq: 2 reverse implication mild gen}
	(G \cdot \nabla_\x)^2 \nabla_\xi^2 \phi = ((G \cdot \nabla_\x)^2 c) B(\nabla_\xi \phi, \xi).
\end{align}
By combining \eqref{eq: 1 reverse implication mild gen} and \eqref{eq: 2 reverse implication mild gen}, we find
\begin{align}
	(G \cdot \nabla_\x)^2 \nabla_\xi^2 \phi = \frac{(G \cdot \nabla_\x)^2 c}{(G \cdot \nabla_\x) c}  (G \cdot \nabla_x) \nabla_\xi^2 \phi.
\end{align}
Therefore Bourgain's condition in the sense of Definition \ref{def: Bourgain cond original} holds with $\lambda = \frac{(G \cdot \nabla_\x)^2 c}{(G \cdot \nabla_\x) c}$. 
\end{proof}

\section{Proof of Theorem \ref{thm: geometric characterization of Bourgain}}\label{sec: geometric characterization of Bourgain}

We will need to choose a parameterization of the $\phi$-curves. By a diffeomorphism in $\x$, we may
assume that the $\phi$-curves are transverse to the $t$-slices $\{(x,t) : x \in \R^{n-1}\}$. That is to say $\det(\nabla_x \nabla_\xi \phi) \neq 0$. We may further assume that $0_M = 0 \in \R^n$ and by a constant rescaling, the $\phi$-curves are defined for $t \in [-1,1]$ (after passing to compact subsets of $M, \Sigma, \V$). Finally by a translation in $\xi$ and after replacing $\phi$ by $\phi(\x,\xi) - \phi(0,\xi)$, we may assume that $0_\Sigma = 0 \in \R^{n-1}$ and $\nabla_\xi \phi(0,0) = 0$. Let $X(\xi,v,t)$ be the unique solution in $x$ to
\begin{align}
    \nabla_\xi \phi(x,t,\xi) = v.
\end{align}
The curves $\ell_{\xi,v}$ are thus given by 
\begin{align}\label{eq: parameterization}
    \ell_{\xi,v} = \{(X(\xi,v,t),t) : |t| \leq 1\},
\end{align}
and we will write $\ell_{\xi,v}(t) = (X(\xi,v,t),t)$. By taking a derivative of the equation $\nabla_\xi \phi(\ell_{\xi,v}(t), \xi)= v$ in $t$, we find $(\partial_t \ell_{\xi,v}(t) \cdot \nabla_\x)\nabla_\xi \phi(\ell_{\xi,v}(t),\xi) =0$. Thus 
\begin{align} \label{eq: integral curve equation}
    \partial_t \ell_{\xi,v}(t) = \lambda(\xi,v,t) G(\ell_{\xi,v}(t),\xi),
\end{align}
for some smooth nonzero scalar function $\lambda(\xi,v,t)$, where $G$ is the Gauss map defined earlier. We will use this fact below. 

We record the following implicit derivative calculations of $X$ for future use: 
\begin{align} 
    \nabla_v X(\xi,v,t) &= \nabla_x \nabla_\xi \phi(X(\xi,v,t),t,\xi)^{-1}, \label{eq: nabla v X} \\
    \nabla_\xi X(\xi,v,t) &= - \nabla_v X(\xi,v,t) \nabla_\xi^2 \phi(X(\xi,v,t), t,\xi). \label{eq: nabla xi X}
\end{align}
Since $\nabla_x \nabla_\xi \phi$ is invertible, \eqref{eq: nabla v X} shows $\nabla_v X(\xi,v,t)$ is invertible too.

\subsection{Forward implication}

\begin{proof}
Let $\phi$ be a phase function satisfying Bourgain's condition. By Proposition \ref{prop: mild reformulation Bourgain} there are $A,B,c$ with $B$ nondegenerate and $(G \cdot \nabla_\x) c \neq 0$, such that \eqref{eq: ABc equation} holds. Fix $(\xi_0, v_0) \in \Sigma \times \V$. We will find the data $F = F_{\xi_0,v_0},\Xi = \Xi_{\xi_0,v_0},V = V_{\xi_0,v_0}$ such that \eqref{eq: bourgain characterization} holds. We may Taylor expand $X$ to first order in $(\xi, v)$ near $(\xi_0,v_0)$ to get 
\begin{align}\label{eq: taylor expansion}
    X(\xi, v, t) = X(\xi_0, v_0, t) + \nabla_\xi X(\xi_0, v_0, t)(\xi - \xi_0) + \nabla_v X(\xi_0, v_0, t)(v - v_0) + O(r^2),
\end{align}
where $r:=|(\xi,v)-(\xi_0,v_0)|$. 
We now relate the terms $\nabla_\xi X(\xi_0, v_0, t)$ and $\nabla_v X(\xi_0, v_0, t)$ to the data $A,B,c$. Define 
\begin{align}
    \Xi(\xi) &:= B(v_0,\xi_0)(\xi-\xi_0), \\
    V(\xi,v) &:= (v - v_0) - A(v_0,\xi_0)(\xi-\xi_0).
\end{align}
The map $\Xi$ is a diffeomorphism since $B(v_0,\xi_0)$ is nondegenerate. Then by \eqref{eq: nabla xi X} and \eqref{eq: ABc equation},
\begin{align}
    X(\xi,v,t) &= X(\xi_0,v_0,t) + \nabla_v X(\xi_0,v_0,t)(V(\xi,v) - c(X(\xi_0,v_0,t),t, \xi_0)\Xi(\xi)) + O(r^2).
\end{align}
We will successively apply three diffeomorphisms to straighten the family $\ell_{\xi,v}$ into a family of lines, up to error $O(|(\xi,v)-(\xi_0,v_0)|^2)$. The first is a recentering, the second is a twist in each $t$-slice, and the third is a diffeomorphism in the $t$-component. 

\noindent \textbf{Step 1: Recentering.}

Define the diffeomorphism $F_1(x,t) = (x - X(\xi_0,v_0,t),t)$. Then 
\begin{align}
    \ell^1_{\xi,v} &:= F_1(\ell_{\xi,v}) \\
    &= \{(\nabla_v X(\xi_0,v_0,t) (V(\xi,v) - c(X(\xi_0,v_0,t),t,\xi_0)\Xi(\xi)) + O(r^2),t) : |t| \leq 1\}.
\end{align}

\noindent \textbf{Step 2: Twist in each $t$-slice.}

Define the diffeomorphism $F_2(x,t) = (\nabla_v X(\xi_0,v_0,t)^{-1}x,t)$. Then 
\begin{align}
    \ell^2_{\xi,v} &:= F_2(\ell^1_{\xi,v}) \\
    &= \{(V(\xi,v)-c(X(\xi_0,v_0,t),t, \xi_0)\Xi(\xi) + O(r^2), t) : |t| \leq 1\}.
\end{align}

\noindent \textbf{Step 3: Diffeomorphism in the $t$-component.}

Define $\tilde c(t) := c(\ell_{\xi_0,v_0}(t),\xi_0)=c(X(\xi_0,v_0,t),t,\xi_0)$. Since $\partial_t \ell_{\xi_0,v_0}(t)$ is a nonzero multiple of $G(\ell_{\xi_0,v_0}(t),\xi_0)$, $\tfrac{d}{dt} \tilde c(t)$ is a nonzero multiple of 
\begin{align}
    (G(\ell_{\xi_0,v_0}(t),\xi_0) \cdot \nabla_\x) c(\ell_{\xi_0,v_0}(t),\xi_0) \neq 0,
\end{align}
and hence $\tfrac{d}{dt} \tilde c(t) \neq 0$. After a constant rescaling, we may assume that $\tilde c$ is a diffeomorphism when restricted to $[-1,1]$. Thus the map $F_3(x,t) = (x,\tilde c(t))$ is a diffeomorphism. We compute 
\begin{align}
    \ell^3_{\xi,v} &:= F_3(\ell^2_{\xi,v}) \\
    &= \{(V(\xi,v) - \tilde c(t) \Xi(\xi)+O(r^2), \tilde c(t)) : |t| \leq 1\} \\
    &= \{(V(\xi,v) - s\Xi(\xi) + O(r^2), s) : |s| \lesssim 1\} \\ 
    &\subset \mathrm{line}_{\Xi(\xi), V(\xi,v)} + O(r^2). 
\end{align}

Collecting the steps above, we can take 
\begin{align}
    F_{\xi_0,v_0}(x,t) &= (\nabla_v X(\xi_0,v_0,t)^{-1}(x - X(\xi_0,v_0,t)), c(X(\xi_0,v_0,t),t,\xi_0)),\\
    \Xi_{\xi_0,v_0}(\xi) &= B(v_0,\xi_0) (\xi-\xi_0), \\
    V_{\xi_0,v_0}(\xi,v) &= (v - v_0) - A(v_0,\xi_0)(\xi-\xi_0),
\end{align}
to conclude the forward direction of Theorem \ref{thm: geometric characterization of Bourgain}. Indeed, $\nabla_{v} V_{\xi_0,v_0}$ and $\nabla_\xi \Xi_{\xi_0,v_0}$ are invertible so 
\begin{align}
    \nabla_{\xi,v}(\Xi_{\xi_0,v_0}, V_{\xi_0,v_0}) = 
    \begin{pmatrix}
        \nabla_\xi \Xi_{\xi_0,v_0} &0 \\
        \nabla_\xi V_{\xi_0,v_0} & \nabla_v V_{\xi_0,v_0}
    \end{pmatrix}
\end{align}
is invertible. 
\end{proof}

\subsection{Reverse implication}\label{subsubsec: Bourgain characterization reverse implication}
\begin{proof}
Fix $(\xi_0,v_0) \in \Sigma \times \V$ and consider the data $F=F_{\xi_0,v_0}$, $\Xi=\Xi_{\xi_0,v_0}$, and $V=V_{\xi_0,v_0}$ in Theorem \ref{thm: geometric characterization of Bourgain}. In the proof, we will not assume that $\nabla_v V$ is invertible a priori, and instead derive this from the other properties of the data in order to address Remark \ref{remark: invertibility follows}.
Write $F(x,t) = (H(x,t), h(x,t)) \in \R^{n-1} \times \R$. 
At the level of parameterizations, \eqref{eq: bourgain characterization} says that
\begin{align}\label{eq: reverse impl line eq 1}
    H(X(\xi,v,t),t) &= V(\xi,v) - h(X(\xi,v,t),t)\Xi(\xi) + O(|(\xi,v)-(\xi_0,v_0)|^2).
\end{align}

\noindent \textbf{Step 1: Renormalizing}

Define the renormalized maps
\begin{align}
    \tilde X(\xi,v,t) &= X(\xi_0+\xi, v_0+v, t) - X(\xi_0,v_0,t), \label{eq: tilde X}\\
    \tilde H(x,t) &= H(X(\xi_0,v_0,t)+x,t) - (V(\xi_0,v_0) -h(X(\xi_0,v_0,t)+x,t))\Xi(\xi_0),\label{eq: tilde H} \\
    \tilde h(x,t) &= h(X(\xi_0,v_0,t)+x,t), \label{eq: tilde h}\\
    \tilde \Xi(\xi) &= \Xi(\xi_0 + \xi)-\Xi(\xi_0), \label{eq: tilde Xi}\text{ and} \\
    \tilde V(\xi,v) &= V(\xi_0 + \xi, v_0+v) -V(\xi_0,v_0). \label{eq: tilde V} 
\end{align}
These maps satisfy
\begin{align}
    \tilde X(0,0,t),\tilde H(0,t), \tilde \Xi(0), \tilde V(0,0) =0.
\end{align}
By \eqref{eq: nabla v X} and \eqref{eq: nabla xi X}, $\nabla_v \tilde X(\xi,v,t)$ is invertible and
\begin{align}\label{eq: nabla xi tilde X}
    \nabla_\xi \tilde X(\xi,v,t) &= -\nabla_v \tilde X(\xi,v,t) \nabla_\xi^2 \phi(X(\xi_0,v_0,t)+\tilde X(\xi,v,t),t,\xi_0 + \xi),
\end{align}
which we will make use of later. 

We claim that $\tilde F(x,t) := (\tilde H(x,t), \tilde h(x,t))$ is a local diffeomorphism. To see this, notice that $\tilde F = F_2 \circ F \circ F_1$, where 
\begin{align}
    F_1(x,t) &= (X(\xi_0,v_0,t)+x,t)
    \ \text{and} \\
    F_2(x,t) &= (x-(V(\xi_0,v_0)-t\Xi(\xi_0)), t).
\end{align}
The map $F$ is a local diffeomorphism by hypothesis, and $F_1,F_2$ can be seen to be local diffeomorphisms after a quick calculation of their Jacobians. Thus $\tilde F$ is a local diffeomorphism. Since $\tilde H(0,t) = 0$, we find $\partial_t \tilde H(0,t) = 0$ and thus  
\begin{align}
    \nabla_{\x}\tilde F(0,t) = 
    \begin{pmatrix}
        \nabla_x \tilde  H(0,t) & 0 \\
        \nabla_x \tilde h(0,t) & \partial_t \tilde h(0,t)
    \end{pmatrix}.
\end{align}
Since this matrix is block upper triangular, $\det(\nabla_\x\tilde F(0,t)) = \det(\nabla_x \tilde H(0,t)) \partial_t \tilde h(0,t)$. As $\det(\nabla_\x\tilde F(0,t)) \neq 0$, 
\begin{align}
    \det(\nabla_x \tilde H(0,t)) &\neq 0 \label{eq: nabla x tilde H invertible} \text{ and}\\
    \partial_t \tilde h(0,t) &\neq 0 \label{eq: partial t tilde h nonzero}.
\end{align}
We will use both these facts below as well.

Subtracting $(V(\xi_0,v_0)-h(X(\xi_0,v_0,t) + x,t))\Xi(\xi_0)$ from both sides of \eqref{eq: reverse impl line eq 1} and substituting in \eqref{eq: tilde X} and \eqref{eq: tilde V}, we get 
\begin{align}\label{eq: reverse impl line eq 2}
    \tilde H(\tilde X(\xi,v,t),t) &= \tilde V(\xi,v) - \tilde h(\tilde X(\xi,v,t),t) \tilde \Xi(\xi) + O(|(\xi,v)|^2). 
\end{align}

\noindent \textbf{Step 2: Taylor expanding}

Compute the following Taylor expansions:
\begin{align}
    \tilde X(\xi,v,t) &= \nabla_\xi \tilde X(0,0,t)\xi + \nabla_v \tilde X(0,0,t)v+O(|(\xi,v)|^2), \\
    \tilde H(x,t) &= H_0(t)x + O(|x|^2), \\
    \tilde h(x,t) &= h_0(t) + O(|x|), \\
    \tilde \Xi(\xi) &= \Xi_0 \xi + O(|\xi|^2), \\
    \tilde V(\xi,v) &= V_0 \xi + V_1 v + O(|(\xi,v)|^2),
\end{align}
where we defined 
\begin{align}
    H_0(t) &:= \nabla_x \tilde H(0,t) : \R \to \R^{(n-1) \times (n-1)}, \label{eq: H0}\\
    h_0(t) &:= \tilde h(0,t) : \R \to \R, \label{eq: h0}\\
    \Xi_0 &:= \nabla_\xi \tilde \Xi(0) \in \R^{(n-1) \times (n-1)},\label{eq: Xi0}  \\
    V_0 &:= \nabla_\xi \tilde V(0,0) \in \R^{(n-1) \times (n-1)}, \label{eq: V0}\\
    V_1 &:= \nabla_v \tilde V(0,0) \in \R^{(n-1) \times (n-1)}. \label{eq: V1}
\end{align}
The equation \eqref{eq: reverse impl line eq 2} then reads 
\begin{align}
    H_0(t)(&\nabla_\xi \tilde X(0,0,t)\xi + \nabla_v \tilde X(0,0,t)v) \label{eq: ungrouped}
    \\&= V_0 \xi + V_1 v - h_0(t)\Xi_0 \xi + O(|(\xi,v)|^2). \nonumber
\end{align}

\noindent \textbf{Step 3: Grouping terms}

Grouping terms in \eqref{eq: reverse impl line eq 2} by $\xi$ and $v$,
\begin{align}\label{eq: reverse impl grouped}
    (H_0(t)&\nabla_\xi \tilde X(0,0,t)-V_0+h_0(t)\Xi_0)\xi \\
    &+ (H_0(t)\nabla_v \tilde X(0,0,t)-V_1)v = O(|(\xi,v)|^2).\nonumber
\end{align}
By taking a $\xi$ derivative of \eqref{eq: reverse impl grouped} and setting $(\xi,v) = (0,0)$, and similarly taking a $v$ derivative of \eqref{eq: reverse impl grouped} and setting $(\xi,v)=(0,0)$, we find 
\begin{align}
    H_0(t) \nabla_\xi \tilde X(0,0, t) - V_0 + h_0(t) \Xi_0 &= 0, \label{eq: 1}\\
    H_0(t) \nabla_v\tilde X(0,0,t) - V_1 &= 0. \label{eq: 2}
\end{align}
Since $H_0(t)= \nabla_x \tilde H(0,t)$ is invertible from \eqref{eq: nabla x tilde H invertible} and $\nabla_v \tilde X$ is invertible, \eqref{eq: 2} shows that $V_1$ is invertible. This addresses Remark \ref{remark: invertibility follows}. Solving for $H_0(t)$ in \eqref{eq: 2} and
plugging the result into \eqref{eq: 1}, we get 
\begin{align}
    V_1 \nabla_v \tilde X(0,0,t)^{-1} \nabla_\xi \tilde X(0,0,t) - V_0+h_0(t) \Xi_0 =0.
\end{align}
Using the invertibility of $V_1$ and the formula \eqref{eq: nabla xi tilde X}, we arrive at 
\begin{align} \label{eq: 3}
    \nabla_\xi^2 \phi(X(\xi_0,v_0,t),t,\xi_0) = -V_1^{-1} V_0 + h_0(t)V_1^{-1} \Xi_0.
\end{align}
We also record that $\partial_t h_0(t) \neq 0$ from \eqref{eq: partial t tilde h nonzero}.

\noindent \textbf{Step 4: Conclusion}

We now use the work above to verify the hypotheses of the reverse direction of Proposition \ref{prop: mild reformulation Bourgain}, in order to conclude that $\phi$ satisfies Bourgain's condition.
By expanding the definitions of $h_0(t)$, $\Xi_0$, $V_0$, and $V_1$, \eqref{eq: 3} becomes 
\begin{align}\label{eq: messy xixi phi expression}
    \nabla_\xi^2 \phi(\ell_{\xi_0,v_0}(t), \xi_0) &= -\nabla_v V_{\xi_0, v_0}(\xi_0,v_0)^{-1} \nabla_\xi V_{\xi_0,v_0}(\xi_0,v_0) \\
    &+h_{\xi_0,v_0}(\ell_{\xi_0,v_0}(t)) \nabla_v V_{\xi_0,v_0}(\xi_0,v_0)^{-1} \nabla_\xi \Xi_{\xi_0,v_0}(\xi_0), \nonumber
\end{align}
and this holds for any choice of $(\xi_0,v_0) \in \Sigma\times \V$.
Define
\begin{align}
    A(v_0,\xi_0) &= -\nabla_v V_{\xi_0,v_0}(\xi_0,v_0)^{-1} \nabla_\xi V_{\xi_0,v_0}(\xi_0,v_0), \\
    B(v_0,\xi_0) &= \nabla_v V_{\xi_0,v_0}(\xi_0,v_0)^{-1} \nabla_\xi \Xi_{\xi_0,v_0}(\xi_0),\\
    c(\x, \xi_0) &= h_{\xi_0,\nabla_{\xi} \phi(\x,\xi_0)}(\x).
\end{align}
First we prove that \eqref{eq: ABc equation} holds. Fix $(\x,\xi_0) \in M \times \Sigma$, where again $\x = (x,t)$. Taking $(\xi_0,v_0,t) = (\xi_0, \nabla_\xi \phi(\x, \xi_0), t)$ in \eqref{eq: messy xixi phi expression}, we conclude
\begin{align}
    \nabla_\xi^2 \phi(\x,\xi_0) = A(\nabla_\xi \phi(\x,\xi_0),\xi_0) +c(\x,\xi_0) B(\nabla_\xi \phi(\x,\xi_0), \xi_0).
\end{align}
Next we verify that $(G(\x,\xi_0) \cdot \nabla_\x) c(\x,\xi_0) \neq 0$. Again choose $v_0 = \nabla_\xi \phi(\x,\xi_0)$ so $\ell_{\xi_0,v_0}(t) = \x$.
We have $c(\ell_{\xi_0,v_0}(s), \xi_0) = h_{\xi_0, v_0}(\ell_{\xi_0,v_0}(s)) $. Taking a derivative in $s$, using \eqref{eq: integral curve equation}, and using $\partial_t h_0(s) \neq 0$ recorded in Step 3, we find 
\begin{align}
    (G(\ell_{\xi_0,v_0}(s), \xi_0) \cdot \nabla_\x) c(\ell_{\xi_0,v_0}(s), \xi_0) \neq 0.
\end{align}
Plugging in $s = t$, we conclude $(G(\x,\xi_0) \cdot \nabla_\x) c(\x,\xi_0) \neq 0$. Finally $B(v_0,\xi_0)$ is invertible since it is a product of invertible matrices. Thus $\phi$ satisfies Bourgain's condition by Proposition \ref{prop: mild reformulation Bourgain}. 
\end{proof}

\section{Discretizing sticky $\phi$-Kakeya sets and Proposition \ref{prop: discretized implies dimension}} \label{sec: discretizing}

In this section, we only describe the mild adjustments to \cite[Section 2.3]{wangZahlstickyKakeya} which are needed for the standard discretization arguments to carry over to the curved case. 

To analyze families of $\phi$-curves with packing dimension close to $n-1$, similarly to \cite{wangZahlstickyKakeya}, define for each $\rho > 0$ and $t > 0$ the class of 
``Quantitatively Sticky'' sets $\mathrm{QStick}(t,\rho)$ to be the collection of sets $L \subset \mathcal C(\phi)$ that satisfy 
\begin{align}
    |\{v \in \V : (v,\xi) \in N_\delta(L)\}| \leq \delta^{n-1-t} \text{ for all } \xi \in \Sigma, \delta \in (0,\rho).
\end{align}
The next lemma is the analogue of \cite[Lemma 2.2]{wangZahlstickyKakeya}.
\begin{lemma}\label{lem: sticky interp}
    Let $L \subset \mathcal C(\phi)$ with $|\dir(L)| > 0$. Then for all $t > \dim_\P(L) - (n-1)$, there exists $L' \subset L$ and $\rho > 0$ with $|\dir(L')| > 0$ and $L' \in \mathrm{QStick}(t,\rho)$. 
\end{lemma}
The proof of Lemma \ref{lem: sticky interp} is exactly the same as the proof of \cite[Lemma 2.2]{wangZahlstickyKakeya}, after replacing $\underline{p}$ with $v$, replacing $v$ with $\xi$, and replacing the metric space $\mathcal L_n$ with the metric space $\mathcal C(\phi)$. 
Then Lemma \ref{lem: sticky interp} is used to prove Proposition \ref{prop: discretized implies dimension} exactly as \cite[Lemma 2.2]{wangZahlstickyKakeya} is used to prove \cite[Proposition 2.1]{wangZahlstickyKakeya}. The necessary adjustments are, again, to replace $\mathcal L_n$ with $\mathcal C(\phi)$. One should also use that the statement 
\begin{align}
    \forall \eps'>0 \exists \eta',\delta_0>0: \SK(\phi,\eps',\eta', \delta_0)
\end{align}
translates to the statement $\sigma_n = 0$ in the case of lines in \cite{wangZahlstickyKakeya}.

\section{Proof of Theorem \ref{thm: sticky reduction formal}} \label{sec: sticky reduction}

We begin by giving some definitions and notation that will be used in this section. 
For $A,B\geq 0$, we will use the notation $A \lesssim B$ to mean that there exists a constant $C > 0$ such that $A \leq CB$, where $C$ depends on at most the diameters of $M,M_0, \V, \V_0, \Sigma,\Sigma_0,$ and upper and lower bounds of a fixed number of derivatives of $\phi$ over $M_0 \times \Sigma_0$. In practice, $5$ derivatives certainly suffice. We write $A \gtrsim B$ to mean there exists such a constant $C$ such that $A \geq C^{-1}B$. We write $A \sim B$ to mean $A \lesssim B$ and $A \gtrsim B$.

It will be useful for induction on scales to consider a version of Definition \ref{def: SK assertion} with $(\delta^{\kappa},\phi)$-tubes, for $\kappa \in (0,1]$. Recall that we assume the tubes $T = T^\delta_{\xi,v}$ satisfy $(\xi,v) \in \Sigma_0 \times \mathcal V_0$. 
\begin{definition}[Assertion $\SK'(\phi, \eps, \eta, \delta_0, \kappa)$]\label{def: SK' assertion}
    Let $\eps, \eta, \delta_0 > 0$, $\kappa \in (0,1]$.
    The assertion $\SK'(\phi, \eps, \eta, \delta_0,\kappa)$ holds if for all $\delta \in (0, \delta_0]$ one has the following. Set $\tau = \delta^\kappa$. For all pairs $(\T, Y)_\tau$ of $(\tau, \phi)$-tubes and their shadings satisfying
    \begin{enumerate}
        \item[(a)] the tubes in $\T$ are essentially distinct, 
        \item[(b)] for each $\tau \leq \rho \leq 1$, $\T$ can be covered by a set of $(\rho,\phi)$-tubes, at most $\tau^{-\eta}$ of which are essentially parallel to a common tube, and
        \item[(c)] $\sum_{T \in \T} |Y(T)| \geq \tau^\eta$,
    \end{enumerate}
    one has 
    \begin{align}
        |\bigcup_{T \in \T}  Y(T)| \geq \delta^{\eps}.
    \end{align}
\end{definition}
The statement $\mathrm{SK'}(\phi, \eps,\eta, \delta_0, 1)$ is equivalent to $\mathrm{SK}(\phi, \eps, \eta, \delta_0)$. We will start at a value of $\kappa$ near $0$ and iterate $\kappa$ up to $1$. The parameters in the statements of Lemmas \ref{lem: base case} and \ref{lem: inductive step} below implicitly depend on a fixed number of derivatives of $\phi$ over $M_0 \times \Sigma_0$ and the diameters of the sets $M,M_0,\V,\V_0,\Sigma,\Sigma_0$. The base case is as follows. 
\begin{lemma}\label{lem: base case}
    Fix $\eps > 0$. There exists $\eta =\eta_{\mathrm{base}} >0$, $\delta_0 = \delta_{0, \mathrm{base}}(\eps)>0$, and $\kappa = \kappa_{\mathrm{base}}(\eps) > 0$ a power of $2$, such that $\mathrm{SK'}(\phi, \eps, \eta, \delta_0, \kappa)$ holds.
\end{lemma}
\begin{proof}
    Take $\eta = 1$ and $\kappa, \delta_0$ to be fixed later. Let $(\T,Y)_\tau$ satisfy (a),(b),(c) in the hypothesis of $\mathrm{SK'}(\phi, \eps, \eta, \delta_0, \kappa)$. By (a), we have the bound $|\T| \lesssim \tau^{-2(n-1)}$. By (c), there is a tube $T \in \T$ with $|Y(T)| \geq |\T|^{-1} \tau \gtrsim \tau^{2n}$. Thus 
    \begin{align}\label{eq: base case 1}
        |\bigcup_{T \in \T} Y(T)| \geq |Y(T)| \gtrsim \tau^{2n}.
    \end{align}
    If we choose $\kappa$ to be the largest power of 2 less than $\eps / (4n)$ and take $\delta_0$ small enough depending on $\eps$ and the implied constant in \eqref{eq: base case 1}, we find 
    \begin{align}
        |\bigcup_{T \in \T} Y(T)| \geq \delta^{\eps}.
    \end{align}
\end{proof}
We will pass from $\kappa = \kappa_{\mathrm{base}}(\eps)$ to $\kappa = 1$ by iterating the following lemma. 
\begin{lemma}\label{lem: inductive step}
    Suppose that $\phi$ satisfies Bourgain's condition and Conjecture \ref{conj: discr classical sticky Kakeya conj} holds in $\R^n$. Then for every $\eps, \eta, \delta_0> 0$, $\kappa \in (0,1/2]$, there exists $\tilde \eta = \eta_{\mathrm{ind}}(\eta)>0$ and  $\tilde \delta_0 = \delta_{0, \mathrm{ind}}(\eta, \delta_0, \kappa)>0$ such that 
    \begin{align}
        \mathrm{SK}'(\phi, \eps, \eta, \delta_0, \kappa) \implies \mathrm{SK}'(\phi, \eps, \tilde \eta, \tilde \delta_0, 2\kappa).
    \end{align}
\end{lemma}

\begin{proof}
    \noindent \textbf{Step 1: Applying Theorem \ref{thm: geometric characterization of Bourgain} and setup}

    Since Conjecture \ref{conj: discr classical sticky Kakeya conj} holds in dimension $n$, we have the constants $\eta'=\eta_{n,\mathrm{rest}}(\eta/100)$ and $\delta_0' = \delta_{0,n,\mathrm{rest}}(\eta/100)$ so that $\SK(\phi_{n,\mathrm{rest}}, \eta/100, \eta', \delta'_0)$ holds. Define $\tilde \eta = \min(\eta, \eta')/100$. We start by assuming $\tilde \delta_0 \leq \min(\delta_0, \delta_0')$, but we will further shrink $\tilde \delta_0$ depending on the allowed parameters in the course of the argument. 

    Since $\phi$ satisfies Bourgain's condition, we can apply Theorem \ref{thm: geometric characterization of Bourgain} to obtain the maps 
    $F_{\xi_0,v_0},\Xi_{\xi_0,v_0},V_{\xi_0,v_0}$. Assume the normalizations in Remark \ref{remark: data normalizations} as well. 
    Since the map $F_{\xi_0,v_0}$ is a smooth local diffeomorphism which varies smoothly in the parameter $(\xi_0,v_0)$, depends only on $\phi$, and the sets $M_0, \Sigma_0, \V_0$ are compact, we have the bounds 
    \begin{align}\label{eq: uniform bounds F}
        |\partial^\alpha F_{\xi_0,v_0}(\x)| \lesssim 1 \text{ and }
        |\det(\nabla_\x F_{\xi_0,v_0}(\x))| \sim 1 ,
    \end{align}
    uniformly over $(\x,\xi_0, v_0) \in M_0 \times \Sigma_0 \times \V_0$ and $|\alpha|\leq 3$.
    Similarly, we have the bounds 
    \begin{align}\label{eq: uniform bounds V and Xi}
        &|\partial^\alpha \Xi_{\xi_0,v_0}(\xi)|, |\partial^\alpha V_{\xi_0,v_0}(\xi,v)| \lesssim 1 \text{ and}
        \\& 
        |\det(\nabla_\xi\Xi_{\xi_0,v_0}(\xi))|, |\det(\nabla_v V_{\xi_0,v_0}(\xi,v))| \sim 1,
    \end{align}
    uniformly over $(\xi_0,v_0, \xi, v) \in (\Sigma_0 \times \V_0)^2$ and $|\alpha|\leq 3$.
    Recall that the coaxial curve of $T$ has parameters in $\Sigma_0 \times \V_0$ and $Y(T) \subset M_0$, for each $T \in \T$. 

    Since we have these uniform bounds, we can apply the inverse function theorem to find $\eps_0 \sim 1$ such that for any $(\xi_0, v_0) \in \Sigma_0 \times \V_0$ and $\x_0 \in \ell_{\xi_0,v_0} \cap M_0$, the following restricted maps, which we will work with for the remainder of the argument, are diffeomorphisms onto their images: 
    \begin{align}
        F_{\xi_0, v_0} &: B^n(\x_0, \eps_0) \to \R^n, \\
        \Xi_{\xi_0,v_0} &: B^{n-1}(\xi_0, \eps_0) \to \R^{n-1}, \\
        (\Xi_{\xi_0,v_0}, V_{\xi_0,v_0}) &: B^{n-1}(\xi_0,\eps_0) \times B^{n-1}(v_0,\eps_0) \to \R^{2(n-1)}.
    \end{align}
    Then using \eqref{eq: bourgain characterization} in Theorem \ref{thm: geometric characterization of Bourgain}, there exists $C,c \sim 1$ such that if $\delta \leq c$ and $T^{\delta^{1/2}}_{\xi_0,v_0}$ covers $T^{\delta}_{\xi,v}$, then $F_{\xi_0,v_0}(T^\delta_{\xi,v} \cap M_0) \subset T^{C\delta}_{\mathrm{line}, \Xi_{\xi_0,v_0}(\xi),V_{\xi_0,v_0}(\xi,v)}$. 
    For each $(\xi_0,v_0)$, we apply a 
    harmless dilation by $1/C$ around $\mathrm{line}_{\Xi_{\xi_0,v_0}(\xi_0),V_{\xi_0,v_0}(\xi_0,v_0)} = \mathrm{line}_{0,0}$, adjusting $F_{\xi_0,v_0}, \Xi_{\xi_0,v_0}$, and $V_{\xi_0,v_0}$ accordingly, to get the cleaner inclusion 
    \begin{align}\label{eq: curved tubes to straight tubes}
        F_{\xi_0,v_0}(T_{\xi,v}^\delta \cap M_0) \subset T^\delta_{\mathrm{line}, \Xi_{\xi_0,v_0}(\xi),V_{\xi_0,v_0}(\xi,v)}.
    \end{align}

    Let $(\T,Y)_\tau$ satisfy (a),(b),(c) in the hypothesis of $\mathrm{SK}'(\phi, \eps, \tilde \eta, \tilde \delta_0, 2\kappa)$, where we start by assuming $\tilde \delta_0 \leq \delta_0$. We will add more restrictions to $\tilde \eta$ and $\tilde \delta_0$ during the course of the argument. 
    Define $\rho = \tau^{1/2}$. By (b), we may cover $\T$ by a set of 
    $(\rho,\phi)$-tubes $\T_\rho$, at most $\tau^{-\tilde \eta}$ of which are parallel to a common tube. 
    Thus $|\T_\rho| \lesssim \tau^{-\tilde \eta} \rho^{-(n-1)}$. Applying (b) at scale $\tau$, we have $|\T[T_\rho]| \lesssim \tau^{-\tilde \eta} (\rho/\tau)^{n-1}$.
    Defining
    \begin{align}
        \T_\rho' = \{T_\rho\in\T_\rho : \sum_{T\in\T[T_\rho]} |Y(T)| \geq c_0 \tau^{2\tilde \eta}\rho^{n-1}\}
    \end{align}
    with $c_0 \gtrsim 1$ chosen sufficiently small, we get from (c) that 
    $|\T_\rho'| \gtrsim \tau^{2 \tilde \eta} \rho^{-(n-1)}$.
    Using that at most $\tau^{-\tilde \eta}$ of the tubes in $\T_\rho'$ are parallel to a common tube, we can find $\T_\rho'' \subset \T_\rho'$ with all tubes essentially distinct and 
    \begin{align}
        |\T_\rho''| \gtrsim \tau^{3\tilde \eta} \rho^{-(n-1)}.
    \end{align}
    Below we will begin by studying $\T[T_\rho]$ for each $T_\rho \in \T_\rho''$. 

    \noindent \textbf{Step 2: Applying classical sticky Kakeya to a rescaled and distorted $\T[T_\rho]$}

    Fix $T_\rho \in \T_\rho''$, let $\ell_{\xi_0,v_0}$ be its coaxial $\phi$-curve, and write $F=F_{\xi_0,v_0},\Xi=\Xi_{\xi_0,v_0},V=V_{\xi_0,v_0}$. We work in the $\eps_0$-balls defined in Step 1.
    After decreasing $\delta_{0}'$ if necessary, we may work in the $\eps_0$-balls in parameter space. By pigeonholing, there is a ball $B = B^n(\x_0, \eps_0)$ with $\x_0 \in \ell_{\xi_0, v_0} \cap M_0$ such that 
    \begin{align}
        \sum_{T \in \T[T_\rho]} |Y(T) \cap B| \gtrsim \sum_{T \in \T[T_\rho]} |Y(T)|.
    \end{align}
    For each $T \in \T[T_\rho]$, define the shading $Y'(T) = Y(T) \cap B$. Thus we can work in an $\eps_0$-ball in $M$. 
    Defining
    \begin{align}
        \T[T_{\rho}]' = \{T \in \T[T_\rho] : |Y'(T)| \geq c_1 \tau^{3 \tilde \eta} \tau^{n-1}\}
    \end{align}
    with $c_1 \gtrsim 1$ sufficiently small, we have $|\T[T_\rho]'| \gtrsim \tau^{2\tilde \eta} (\rho/\tau)^{n-1} = \tau^{2\tilde \eta} \rho^{-(n-1)}$. We now use the maps $F$ and $P:=(\Xi,V)$ and \eqref{eq: curved tubes to straight tubes} to define a new collection of straight $\tau$-tubes and shadings: 
    \begin{align}
        \T[T_\rho]'' &:= \{T^\tau_{\mathrm{line}, P(\xi,v)} : T^\tau_{\xi,v} \in \T[T_\rho]'\}, \\
        Y''(T) &:=F(Y'(T^\tau_{P^{-1}(\xi,v)})) \subset T, \text{ for each $T=T^\tau_{\mathrm{line},\xi,v} \in \T[T_\rho]''$}. 
    \end{align}
    Since $F$ is a diffeomorphism with Jacobian $\sim 1$, we have $|Y''(T)| \gtrsim \tau^{3\tilde \eta} \tau^{n-1}$.
    After possibly refining $\T[T_\rho]''$ by a constant factor, we may assume the tubes are essentially distinct (using that $P$ has Jacobian $\sim 1$). Recall that $(\T,Y)$ satisfies (b) and $\T[T_\rho]' \subset \T$, so in particular for each $\tau \leq \rho' \leq \rho$, $\T[T_\rho]'$ can be covered by a set of $(\rho', \phi)$-tubes, at most $\tau^{-\tilde \eta}$ of which are parallel to a common tube. Since $\Xi$ has Jacobian $\sim 1$, the same is true for $\T[T_\rho]''$ after accepting a constant factor more tubes parallel to a common tube. Consider $j(x,t)=(\rho^{-1}x,t)$, the radial dilation by $\rho^{-1}$ around $\mathrm{line}_{0,0}$. We define the collection of straight $\rho$-tubes and shadings: 
    \begin{align}
        \T_1 &:=\{T^{\rho}_{\mathrm{line},\rho^{-1}\xi,\rho^{-1}v} : T_{\xi,v}^\tau \in \T[T_\rho]''\},\label{eq: T1 def} \\
        Y_1(T) &:= j(Y''(T_{\mathrm{line}, \rho \xi,\rho v}^\tau)) \subset T, \text{for each $T = T_{\mathrm{line}, \xi, v}^{\rho} \in \T_1$}.
    \end{align}
    We will now show that $(\T_1, Y_1)$ satisfies the hypotheses (a),(b),(c) in the statement $\mathrm{SK}(\phi_{n,\mathrm{rest}}, \eta/100,\eta', \delta'_0)$, for a small enough $\tilde \delta_0$.
    That (a) holds is clear. That (b) holds 
    follows from the analogous statement for $\T[T_\rho]''$ shown above, $\tilde \eta \leq \eta'/100$, and the parameter rescaling $(\xi,v) \mapsto(\rho^{-1}\xi, \rho^{-1} v)$ in \eqref{eq: T1 def} (and after shrinking $\tilde \delta_0$ slightly to absorb the implicit constant). 
    The rescaling $j$ gives $|Y_1(T)| \gtrsim \rho^{-(n-1)} \cdot \tau^{3 \tilde \eta} \tau^{n-1} = \tau^{3\tilde \eta} \rho^{n-1}$. We also have $|\T_1| \gtrsim |\T[T_\rho]'| \gtrsim \tau^{2\tilde \eta} \rho^{-(n-1)}$. Thus 
    \begin{align}\label{eq: (c) rescaled}
        \sum_{T \in \T_1} |Y_1(T)| \gtrsim \tau^{5\tilde \eta}.
    \end{align}
    After shrinking $\tilde \delta_0$ depending on $\kappa$ and the implicit constant, the left hand side of \eqref{eq: (c) rescaled} is $\geq \rho^{20 \tilde \eta}\geq \rho^{\eta'}$. Thus (a),(b),(c) hold and we may conclude 
    \begin{align}\label{eq: straight lower bd}
        |\bigcup_{T \in \T_1} Y_1(T)| \geq \rho^{\eta/100}.
    \end{align}
    By composing the set in \eqref{eq: straight lower bd} with the maps $j^{-1}$ and $F^{-1}$, we get 
    $|\bigcup_{T \in \T[T_\rho]} Y(T)| \gtrsim \rho^{n-1}|\bigcup_{T\in \T_1} Y_1(T)|$. 
    After shrinking $\tilde \delta_0$ to absorb the implicit constant, 
    \begin{align}\label{eq: union in rho tube}
        |\bigcup_{T \in \T[T_\rho]} Y(T)| \geq \rho^{n-1} \rho^{\eta/5}. 
    \end{align}

    \noindent \textbf{Step 3: Applying sticky $\phi$-Kakeya at scale $\rho$ and conclusion}

    Our goal is to apply $\mathrm{SK'}(\phi, \eps,\eta, \delta_0, \kappa)$ to the pair $(\T_2,Y_2)_\rho$, where 
    \begin{align}
        \T_2 &:= \T_\rho'', \\
        Y_2(T_{\rho}) &:= \bigcup_{T \in \T[T_\rho]} Y(T) \subset T_\rho, \text{ for each $T_\rho \in \T_2$.}
    \end{align}
    We verify that (a),(b),(c) in the hypothesis of that statement hold.
    We have shown (a) holds at the end of Step 1. Now we show (b). Fix
    $\tilde \rho \in [\rho, 1]$. If $\tilde \rho \in [\rho,2\rho]$ the statement is obvious, so we assume $\tilde \rho \in [2\rho,1]$.
    We know $\T$ can be covered by a set of $(\tilde \rho/2, \phi)$-tubes, at most $\tau^{-\tilde \eta}$ are parallel to a common tube. Let $\T_{\tilde \rho}$ be the tubes with the same coaxial curves as in the cover above but with radius $\tilde \rho$. Then $\lesssim \tau^{-\tilde \eta} = \rho^{-2\tilde \eta}$ tubes are parallel to a common tube. 
    That $\T_{\tilde \rho}$ covers $\T_2$ follows directly from the triangle inequality for the metric $d$ on $\phi$-curves. Thus we have verified (b) (after possibly shrinking $\tilde \delta_0$ to absorb the implicit constant and using $\tilde \eta \leq \eta/100$). By \eqref{eq: union in rho tube}, the bound $|\T_2| \gtrsim \tau^{3 \tilde \eta} \rho^{-(n-1)}$, and using $\tilde \eta \leq \eta/100$, 
    \begin{align}
        \sum_{T \in \T_2} |Y_2(T)| \gtrsim \rho^{\eta/2}. 
    \end{align}
    Again shrinking $\tilde \delta_0$ to absorb the implicit constant, the left hand side is $\geq \rho^\eta$. This proves that (c) holds. Finally we may apply the conclusion of $\mathrm{SK'}(\phi, \eps,\eta, \delta_0, \kappa)$ to obtain 
    \begin{align}
        |\bigcup_{T \in \T} Y(T)| \geq |\bigcup_{T_\rho \in \T_2} Y_2(T_\rho)| &\geq \delta^{\eps}.
    \end{align}
\end{proof}
We are now ready to prove Theorem \ref{thm: sticky reduction formal}. 
\begin{proof}[Proof of Theorem \ref{thm: sticky reduction formal}]
Fix $\eps > 0$ and define the sequences $\eta_i, \delta_{0,i}, \kappa_i$ inductively as follows. Set $\eta_1 = \eta_{\mathrm{base}}, \delta_{0,1}=\delta_{0,\mathrm{base}}(\eps),\kappa_1=\kappa_{\mathrm{base}}(\eps)$. Once $\eta_i, \delta_{0,i}, \kappa_i$ have been defined, we set 
\begin{align}
    \eta_{i+1} &= \eta_{\mathrm{ind}}(\eta_i), \\
    \delta_{0,i+1} &= \delta_{0,\mathrm{ind}}(\eta_i, \delta_{0,i}, \kappa_i), \\
    \kappa_{i+1} &= 2\kappa_i.
\end{align}
We define these sequences up to $N = 1+\log_2(1/\kappa_1)$, so $\kappa_N=1$. Applying Lemma \ref{lem: base case} and then iterating Lemma \ref{lem: inductive step} $N-1$ times, we obtain $\SK'(\phi, \eps, \eta_N,\delta_{0,N}, 1)$. Setting $\eta_{\phi}(\eps)=\eta_{N}$ and $\delta_{0,\phi}(\eps) = \delta_{0,N}$, we are done. 
\end{proof}

\section{The $\tan$-example and Proposition \ref{prop: tan example curved}}\label{sec: tan example}

\subsection{Constructing Example \ref{ex: tan example} from Proposition \ref{prop: mild reformulation Bourgain}}

One can build Example \ref{ex: tan example} by choosing a simple ansatz for \eqref{eq: ABc equation}. We first choose a special form of $\phi$ to reduce to solving a system of decoupled ODEs. 
For $1 \leq j \leq n-1$, let $f_j(x_j,t,\xi_j) : \R^3 \to \R$ be functions to be determined later, and define 
\begin{align}
    \phi(\x,\xi) &= \sum_{j=1}^{n-1}\int f_j(x_j,t,\xi_j) d\xi_j.
\end{align}
We also choose the following ansatz for $A$,$B$, and $c$: 
\begin{align}
    A(v,\xi) &= 
    \begin{pmatrix}
        0 & 0 \\
        0 & v_{n-1}^2
    \end{pmatrix},\\
    B(v,\xi) &= I_{n-1},\\
    c(\x,\xi) &= t^2.
\end{align}
Then \eqref{eq: ABc equation} becomes the following decoupled system of $n-1$ ODEs:
\begin{align}
    \partial_{\xi_j} f_j &= t^2 \text{\quad for } 1 \leq j \leq n-2, \\
    \partial_{\xi_{n-1}} f_{n-1} &= f_{n-1}^2 + t^2.
\end{align}
The following choices of $f_j$ solve these equations
\begin{align}
    f_j(x_j,t,\xi_j) &= t^2 \xi_{j} + x_{j}, \\
    f_{n-1}(x_{n-1},t,\xi_{n-1}) &= t \tan(t\xi_{n-1} + x_{n-1}),
\end{align}
and furthermore one may compute $\det \nabla_\x\nabla_\xi \phi(\x,\xi) \neq 0$ for $(\x,\xi)$ in a neighborhood of $((0,1),0)$. 
After integrating the functions $f_j$, we obtain
\begin{align}
    \phi(\x,\xi) = x' \cdot \xi' + \tfrac{1}{2} t^2|\xi'|^2 + \log(\sec(t\xi_{n-1} + x_{n-1})). 
\end{align}
It is also easy to check that $(G(\x,\xi) \cdot \nabla_\x) c(\x,\xi) \neq 0$ for $(\x,\xi)$ in a neighborhood of $((0,1), 0)$. Since $B$ is also positive-definite, Proposition \ref{prop: mild reformulation Bourgain} shows that $\phi$ is a positive-definite phase function satisfying Bourgain's condition. 

\subsection{The family of $\phi_{n,\tan}$-curves}

With $\phi = \phi_{n,\tan}$, we compute 
\begin{align}
    \nabla_\xi \phi(\x,\xi) = (x' + t^2 \xi', t \tan(t \xi_{n-1} + x_{n-1})).
\end{align}
Writing $v = (v',v_{n-1}) \in \R^{n-2} \times \R$, we need to compute $X(\xi,v,t)$, the unique solution $x$ to $\nabla_\xi \phi(x, t,\xi) = v$. That is, 
\begin{align}
    x' + t^2 \xi' &= v', \\
    t \tan(t \xi_{n-1} + x_{n-1}) &= v_{n-1}. 
\end{align}
Solving, we find 
\begin{align}
    x' &= v' - t^2 \xi', \\
    x_{n-1} &= \tan^{-1}(\frac{v_{n-1}}{t}) - t\xi_{n-1},
\end{align}
as long as $|t - 1| \leq 1/10$ and $|\xi|,|v| \leq 1/2$. Thus $X(\xi,v,t) = (v' - t^2\xi', \tan^{-1}(\frac{v_{n-1}}{t}) - t\xi_{n-1})$ and 
\begin{align}
    \ell_{\xi,v} = \{(v' - t^2\xi', \tan^{-1}(\frac{v_{n-1}}{t}) - t\xi_{n-1},t) : |t - 1| \leq 1/10 \}.
\end{align}

\subsection{Illustration of Theorem \ref{thm: geometric characterization of Bourgain} with $\phi_{n,\tan}$}
Before proving Proposition \ref{prop: tan example curved}, let us illustrate the forward direction of Theorem \ref{thm: geometric characterization of Bourgain} for $\phi_{n,\tan}$ by exhibiting a diffeomorphism to lines up to quadratic error. Fix a central curve $\ell_{\xi_0,v_0}$. Write $\xi_0 = (\xi_0',\xi_{0,n-1})$ and $v_0 = (v_0', v_{0,n-1})$. Taylor expanding $\tan^{-1}(v_{n-1}/t)$ to first order in $v_{n-1}$ near $v_{0,n-1}$, we find  
\begin{align}
    \ell_{\xi,v} = \{ 
        &(v' - t^2 \xi', \tan^{-1}(\frac{v_{0,n-1}}{t}) + \frac{t}{t^2 + v_{0,n-1}^2}(v_{n-1} - v_{0,n-1}) - t\xi_{n-1} \\
        &+ O(|v_{n-1}-v_{0,n-1}|^2),t) 
    : |t - 1| \leq 1/10\}.\nonumber
\end{align}
We may first apply the shift in each $t$-slice $F_1(\x) = (x',x_{n-1} - \tan^{-1}(\frac{v_{0,n-1}}{t}),t)$, second apply the shear in each $t$-slice $F_2(\x) = (x', \frac{t^2 + v_{0,n-1}^2}{t} x_{n-1},t)$, and third apply the diffeomorphism in the $t$-direction $F_3(\x) = (x,t^2)$ to get 
\begin{align}
    (F_3 \circ F_2\circ F_1)(\ell_{\xi,v}) &= \{(v' - t^2 \xi', v_{n-1} - v_{0,n-1} - v_{0,n-1}^2\xi_{n-1} - t^2 \xi_{n-1} \\&\ \quad+ 
    (O(|v_{n-1}-v_{0,n-1}|^2), t^2)
     : |t - 1| \leq 1/10 \} \nonumber \\
    &\subset \mathrm{line}_{\Xi(\xi), V(\xi,v)} + O(|(\xi,v) -(\xi_0,v_0)|^2),
\end{align}
where $\Xi(\xi) = \xi$ and $V(\xi,v) = (v', v_{n-1}-v_{0,n-1} - v_{0,n-1}^2\xi_{n-1}))$.

\subsection{Proof of Proposition \ref{prop: tan example curved}}

Fix $\eps_0 > 0$ and assume for contradiction that there is a diffeomorphism $F : B_{\eps_0}(0_M) \to \R^n$ onto its image and functions $\Xi,V$ such that for $(\xi,v)$ near $(0,0)$,
\begin{align}\label{eq: 1 tan curved}
    F(\ell_{\xi,v} \cap B_{\eps_0}(0_M)) \subset \mathrm{line}_{\Xi(\xi,v), V(\xi,v)} + O(|(\xi,v)|^4).
\end{align}
By adjusting the curve $\ell_{\xi,v}$ by an amount $O(|(\xi,v)|^4)$ in each $t$-slice, we obtain a curve $\tilde \ell_{\xi,v}$ satisfying 
\begin{align}\label{eq: 2 tan curved}
    F(\tilde \ell_{\xi,v} \cap B_{\eps_0}(0_M)
    ) \subset \mathrm{line}_{\Xi(\xi,v), V(\xi,v)}.
\end{align}
We will obtain a contradiction as follows. We consider the $1$-parameter family of curves $\tilde \ell_{\xi_{\mathbf p}(s),v_{\mathbf p}(s)}$ passing through $\tilde \ell_{0,0}$ and a point $\mathbf{p} \notin \tilde \ell_{0,0}$ (we will see why $\xi_{\mathbf p}(s), v_{\mathbf p}(s)$ exist locally and are smooth in Step 2 below). We will show that the tangent directions $\gamma(s)$ of the curves $\tilde \ell_{\xi_{\mathbf p}(s),v_{\mathbf p}(s)}$ at $\mathbf{p}$ are \emph{not} all contained in a 2-dimensional vector space. Concretely, we will show that $\gamma(s) \wedge \dot \gamma(s) \wedge \ddot \gamma(s) \neq 0$ for a fixed $s$. 
This is impossible by \eqref{eq: 2 tan curved}, since the lines passing through a fixed line and a point not on that line must be contained in a plane. To make the parameterization of the curves $\tilde \ell_{\xi_{\mathbf p}(s),v_{\mathbf p}(s)}$ easier to compute, we will first simplify the curves $\tilde \ell_{\xi,v}$ via a Taylor expansion and a diffeomorphism. Then we will write down the parameterization, and finally show that the tangent directions through $\mathbf{p}$ are not contained in a $2$-plane. The details are as follows.

\begin{proof}

\noindent \textbf{Step 1: Simplifying the family of curves.}

Each $\tilde \ell_{\xi,v}$ (restricted to $B_{\eps_0}(0_M)$) takes the form 
\begin{align}
    \tilde \ell_{\xi,v} &= \{(X(\xi,v,t) + O(|(\xi,v)|^4) ,t) :  |t - 1| < \eps_0\} \\
    &= \{(v' - t^2 \xi'+O(|(\xi,v)|^4), \tan^{-1}(\frac{v_{n-1}}{t})-t\xi_{n-1} + O(|(\xi,v)|^4),t) : |t-1| < \eps_0\}.
\end{align}
By Taylor expanding in $v_{n-1}$ to order $3$ near 0,
\begin{align}
    \tilde \ell_{\xi,v} &= \{(v' - t^2 \xi' + O(|(\xi,v)|^4), \frac{v_{n-1}}{t} - \frac{v_{n-1}^3}{3t^3} - t\xi_{n-1} + O(|(\xi,v)|^4), t) :|t - 1| < \eps_0 \}.
\end{align}
After replacing $F$ with $F \circ H^{-1}$ and $\tilde \ell_{\xi,v}$ with $H \circ \tilde \ell_{\xi,v}$, where $H$ is the diffeomorphism defined by $H(\x) = (x',t^3 x_{n-1},t^2)$, \eqref{eq: 2 tan curved} continues to hold (with a possibly smaller value of $\eps_0$). Now $\tilde \ell_{\xi,v}$ has the parameterization
\begin{align}\label{eq: tilde l parameterized}
    \tilde \ell_{\xi,v}(t) = (v'-t\xi' + O(|(\xi,v)|^4), t v_{n-1} - \frac{v_{n-1}^3}{3}-t^2 \xi_{n-1} + O(|(\xi,v)|^4), t).
\end{align}

\noindent \textbf{Step 2: The parameters of a curve $\tilde \ell_{\xi,v}$ through a point $\mathbf p$ and $\tilde \ell_{0,0}$.}

Consider the point $\mathbf{p} = (p, t_0)=(p', p_{n-1}, t_0) \in B_{\eps_0}(0_M)$, where $t_0 > 1$. We need to solve for the parameters $\xi_\mathbf{p}(s),v_{\mathbf{p}}(s)$ such that 
\begin{align}\label{eq: 1 step 1 tan curved}
    \tilde \ell_{\xi_{\mathbf{p}}(s), v_{\mathbf{p}}(s)}(s) &= \tilde \ell_{0,0}(s),  \\
    \tilde \ell_{\xi_{\mathbf p}(s), v_{\mathbf p}(s)}(t_0) &= \mathbf p. \label{eq: 1 step 1 tan curved 2}
\end{align}
We first check that $\xi_{\mathbf p}(s), v_{\mathbf p}(s)$ exist locally near $(s, \mathbf p) = (1,0)$ and are smooth. Let $\tilde X(\xi,v,t)$ denote the first $n-1$ coordinates of \eqref{eq: tilde l parameterized} and define $H(\xi,v,s) = (\tilde X(\xi,v,s)-\tilde X(0,0,s), \tilde X(\xi,v,t_0) - p)$. Clearly $H$ is smooth and $H(0,0,1) = 0$. By the implicit function theorem, we have the desired functions if $\nabla_{\xi,v} H(0,0,1)$ is invertible. One may compute 
\begin{align}
    \nabla_{\xi,v} H(0,0,1) &=
    \begin{pmatrix}
        -I_{n-2} & 0 & I_{n-2} & 0 \\
        0 & -1 & 0 & 1\\
        -t_0 I_{n-2} & 0 & I_{n-2} & 0 \\
        0 & -t_0^2 & 0 & t_0
    \end{pmatrix},
\end{align}
which is invertible when $t_0 \neq 1$ (by performing row and column operations for instance).

We consider $s$ near 1 and solve for $\xi_{\mathbf p}(s)$ and $v_{\mathbf p}(s)$ to third order in $p$ near 0 (with $t_0 > 1$ fixed). We have $\xi_0(s)=v_0(s)=0$, so $\xi_{\mathbf p}(s), v_{\mathbf p}(s) = O(|p|)$. Below we will suppress the ``$\mathbf p$'' in $\xi_{\mathbf p}(s)$ and $v_{\mathbf p}(s)$ for ease of notation. 
In coordinates, \eqref{eq: 1 step 1 tan curved} and \eqref{eq: 1 step 1 tan curved 2} give the following system of equations:
\begin{align}
    v'(s) - s\xi'(s) + O(|p|^4) &= 0, \label{eq: prime eq s} \\
    v'(s) - t_0\xi'(s) + O(|p|^4) &= p', \label{eq: prime eq t0} \\
    s v_{n-1}(s) - \frac{v_{n-1}(s)^3}{3}-s^2 \xi_{n-1}(s) + O(|p|^4) &= 0, \label{eq: n-1 eq s} \\
    t_0 v_{n-1}(s) - \frac{v_{n-1}(s)^3}{3}-t_0^2 \xi_{n-1}(s) + O(|p|^4) &= p_{n-1}.\label{eq: n-1 eq t0}
\end{align}
Solving for $\xi'(s)$ using \eqref{eq: prime eq s} and \eqref{eq: prime eq t0}, we get 
\begin{align}
    \xi'(s) &= -\frac{1}{t_0-s}p' + O(|p|^4), \label{eq: xi' expr}\\
    v'(s) &= -\frac{s}{t_0-s}p' + O(|p|^4). \label{eq: v' expr}
\end{align}
Using \eqref{eq: n-1 eq s} to solve for $\xi_{n-1}(s)$ in terms of $v_{n-1}(s)$, 
\begin{align}\label{eq: xi n-1 in terms of v n-1}
    \xi_{n-1}(s) &= \frac{v_{n-1}(s)}{s} - \frac{v_{n-1}(s)^3}{3s^2} + O(|p|^4). 
\end{align}
Substituting \eqref{eq: xi n-1 in terms of v n-1} into \eqref{eq: n-1 eq t0} and simplifying the result, we obtain the following cubic equation in $v_{n-1}(s)$: 
\begin{align}\label{eq: cubic in v n-1}
    \frac{1}{3} (\frac{t_0^2}{s^2} - 1) v_{n-1}(s)^3 - t_0(\frac{t_0}{s} - 1) v_{n-1}(s)- p_{n-1} + O(|p|^4)= 0.
\end{align}
Solving \eqref{eq: cubic in v n-1} iteratively in powers of $p_{n-1}$, we obtain 
\begin{align}
    v_{n-1}(s) &=  -\frac{1}{t_0(\frac{t_0}{s}-1)}p_{n-1}-\frac{1}{3} \frac{(\frac{t_0^2}{s^2} - 1)}{t_0^4(\frac{t_0}{s}-1)^4} p_{n-1}^3 + O(|p|^4) \\
    &= - \frac{s}{t_0(t_0-s)}p_{n-1} - \frac{s^2(t_0+s)}{3t_0^4(t_0-s)^3}p_{n-1}^3 + O(|p|^4). \label{eq: vn-1 expr}
\end{align}
Now we substitute \eqref{eq: vn-1 expr} back into \eqref{eq: xi n-1 in terms of v n-1} to obtain 
\begin{align}\label{eq: xin-1 expr}
    \xi_{n-1}(s) = -\frac{1}{t_0(t_0-s)}p_{n-1} - \frac{s^2}{3t_0^4 (t_0-s)^3}p_{n-1}^3 + O(|p|^4).
\end{align}

\noindent \textbf{Step 3: The curves through $\tilde \ell_{0,0}$ and $\mathbf p \notin \tilde \ell_{0,0}$ are not contained in a surface.}

Define $\gamma(s) = \frac{d}{dt} \tilde \ell_{\xi(s), v(s)}(t_0)$, the tangent vector at point $\mathbf p$ of the curve passing through $\tilde \ell_{0,0}(s)$ and $\mathbf p$. 
We compute 
\begin{align}
    \gamma(s) = &(-\xi'(s)+O(|p|^4),v_{n-1}(s)-2t_0\xi_{n-1}(s) + O(|p|^4),1) \\
    =  & \Big (\frac{1}{t_0-s}p'+O(|p|^4), \\
    &\frac{2t_0-s}{t_0(t_0-s)}p_{n-1} + \frac{s^2}{3t_0^4 (t_0-s)^2} p_{n-1}^3 + O(|p|^4),
    1 \Big ).
\end{align}
We may now compute $\gamma(1), \dot \gamma(1), \ddot \gamma(1)$:
\begin{align}
    \gamma(1) = &\Big (\frac{1}{t_0-1}p' + O(|p|^4),\\
    &\frac{2t_0-1}{t_0(t_0-1)}p_{n-1} + \frac{1}{3t_0^4(t_0-1)^2}p_{n-1}^3 + O(|p|^4),
    1\Big ),\\
    \dot \gamma(1) = &\Big (\frac{1}{(t_0-1)^2} p'+O(|p|^4),\\
    &\frac{1}{(t_0-1)^2}p_{n-1} + \frac{2}{3 t_0^3 (t_0-1)^3}p_{n-1}^3 + O(|p|^4),
    0\Big ), \\
    \ddot \gamma(1) = &\Big (\frac{2}{(t_0-1)^3}p' + O(|p|^4), \\
    &\frac{2}{(t_0-1)^3}p_{n-1} + \frac{2(2+t_0)}{3 t_0^3 (t_0-1)^4}p_{n-1}^3 + O(|p|^4), 
    0\Big ).
\end{align}
Choose $p' = (0,\ldots,0,p_{n-2})$ so that $\gamma(1), \dot \gamma(1), \ddot \gamma(1) \in \{0\}^{n-3} \times \R^3$, and we may view these as vectors in $\R^3$. We may finally compute 
\begin{align}
    |\gamma(1) &\wedge \dot \gamma(1) \wedge \ddot \gamma(1)| = |\det(\gamma(1), \dot \gamma(1), \ddot \gamma(1))| \\
    &= \left |\det 
\begin{pmatrix}
    \frac{1}{(t_0-1)^2} p_{n-2}+O(|p|^4) & \frac{1}{(t_0-1)^2}p_{n-1} + \frac{2}{3 t_0^3 (t_0-1)^3}p_{n-1}^3 + O(|p|^4) \\
    \frac{2}{(t_0-1)^3}p_{n-2} + O(|p|^4) & \frac{2}{(t_0-1)^3}p_{n-1} + \frac{2(2+t_0)}{3 t_0^3 (t_0-1)^4}p_{n-1}^3 + O(|p|^4) 
\end{pmatrix}
\right | \\
&= \frac{2|p_{n-1}|^3|p_{n-2}|}{3 t_0^2 (t_0 -1)^6} + O(|p|^5).
\end{align}
By choosing $p_{n-1}$ and $p_{n-2}$ small enough but nonzero,
\begin{align}\label{eq: wedge nonzero}
    |\gamma(1) \wedge \dot \gamma(1) \wedge \ddot \gamma(1)| > 0.
\end{align}

\noindent \textbf{Step 4: Reaching a contradiction with \eqref{eq: 2 tan curved}}

By \eqref{eq: 2 tan curved}, $F(\tilde \ell_{\xi,v})$ is a family of lines. Thus for $s$ near $1$, $F(\tilde \ell_{\xi(s), v(s)})$ is a family of lines passing through $F(\tilde \ell_{0,0})$ and $F(\mathbf p)$, so it is contained in a fixed plane. In particular the tangent vectors of $F(\tilde \ell_{\xi(s),v(s)})$ at $F(\mathbf p)$ lie in a fixed plane, so $\eta(s) := \frac{d}{dt}F(\tilde \ell_{\xi(s),v(s)}(t_0))$ satisfies $|\eta(1) \wedge \dot \eta(1) \wedge \ddot \eta(1)| = 0$. But by the chain rule 
\begin{align}
    0&=|\eta(1) \wedge \dot \eta(1) \wedge \ddot \eta(1)| = |DF(\mathbf p) \gamma(1) \wedge DF(\mathbf p) \dot \gamma(1) \wedge DF(\mathbf p) \ddot \gamma(1)|,
\end{align}
which contradicts \eqref{eq: wedge nonzero}.
\end{proof}

\appendix

\section{Steps toward a general to sticky reduction for curved Kakeya}\label{sec: appendix}

\subsection{Direction separated tubes and polynomial Wolff axioms: proof of Theorem \ref{thm: dir sep implies PWA}} \label{subsec: dir sep implies PWA}

We recall Theorem \ref{thm: dir sep implies PWA} before giving its proof.
\thmdirsepimpliesPWA*

We will need a more flexible definition of curved tubes in this section. 
\begin{definition}
    Fix a map $\Phi : \R^{n-1} \times \R^{n-1} \times \R \to \R^{n-1}$ which is smooth on a neighborhood of $[-2,2]^{2n-1}$. Fix $\lambda \in (0,1]$, $t_0 \in [-1,1]$, $\delta \in (0,1]$, and define the tubes 
\begin{align}
    T_{\xi,v,\Phi,t_0}(\delta, \lambda) := \{(x,t)\in \R^n : |x-\Phi(\xi,v,t)| \leq \delta, |t-t_0| \leq \lambda\}.
\end{align}
The ``direction'' of $T_{\xi,v,\Phi,t_0}(
\delta,\lambda)$ is $\xi$. When $\lambda =1$, we always assume $t_0=0$ and we will drop this subscript. 
\end{definition}

Fix a phase function $\phi$ in $\R^n$ satisfying Bourgain's condition, with $\phi$-curves $\ell_{\xi,v}$ that have the parameterization $\ell_{\xi,v}(t) = (X(\xi,v,t), t)$ from \eqref{eq: parameterization}. After harmless rescaling, we may assume that $\phi$ and $(\nabla_x \nabla_\xi \phi)^{-1}$ are smooth on $[-2,2]^{2n-1}$.
If $\T$ is a $(\delta/\lambda)$-direction separated collection of tubes $\{T_{\xi_i,v_i,X}(\delta,1)\}_{i}$, the objective is to show
\begin{align}
    |\{T \in \T : T \cap B(\x_0,\lambda) \neq \emptyset \text{ and }T \cap B(\x_0,2\lambda) \subset S\}| \leq C(n,E,\eps)\delta^{-\eps} |S| (\delta^{n-1} \lambda)^{-1},
\end{align}
for $\lambda \in (\delta,1]$, $\x_0 = (x_0,t_0) \in [-1,1]^n$, and $S \subset B(\x_0,4\lambda)$ semi-algebraic of complexity $\leq E$. 
Define $v_{\x_0}(\xi) = \nabla_\xi \phi(\x_0,\xi)$, so $X(\xi,v_{\x_0}(\xi),t_0)=x_0$.
Then $T_{\xi,v,X}(\delta,1) \cap B(\x_0,\lambda) \neq \emptyset$ implies that $T_{\xi,v,X,t_0}(\delta, \lambda) \subset T_{\xi,v,X}(\delta,1) \cap B(\x_0,2\lambda)$ and $|v - v_{\x_0}(\xi)| \leq \lambda$ (after a harmless normalizations of $X$ at the start). 
We therefore get
\begin{align}
    |\{T \in \T &: T \cap B(\x_0,\lambda) \neq \emptyset \text{ and }T \cap B(\x_0,2\lambda) \subset S\}| \leq \\&|\{T_{\xi_i,v_i,X,t_0}(\delta,\lambda) : T_{\xi_i,v_i,X,t_0}(\delta, \lambda) \subset S, |v_i-v_{\x_0}(\xi_i)|\leq \lambda\}|. \label{eq: appendix 1}
\end{align}
Define the rescaling $L_\lambda(\x) = \lambda^{-1}(\x -\x_0)$ which sends $B(\x_0,\lambda)$ to $B(0,1)$, and the rescaled maps
\begin{align}
    X_{\lambda, \x_0}(\xi,v,t) = \lambda^{-1}(X(\xi,v_{\x_0}(\xi)+\lambda v, t_0+\lambda t) - x_0),
\end{align}
which are smooth functions defined on a neighborhood of $[-1,1]^{2n-1}$. 
We get 
\begin{align}
    L_\lambda(T_{\xi,v, X,t_0}(\delta,\lambda)) = T_{\xi,\lambda^{-1}(v-v_{\x_0}(\xi)), X_{\lambda,\x_0}}(\delta/\lambda,1),
\end{align}
and therefore
\begin{align}
    \eqref{eq: appendix 1} = |\{T_{\xi_i,\tilde v_i, X_{\lambda,\x_0}}(\delta/\lambda,1)  : 
    T_{\xi_i,\tilde v_i, X_{\lambda,\x_0}}(\delta/\lambda,1) \subset L_\lambda(S),|\tilde v_i| \leq 1\}|,\label{eq: appendix 2}
\end{align}
where $\tilde v_i :=\lambda^{-1}(v_i -v_{\x_0}(\xi_i))$. Using that $|L_\lambda(S)| = \lambda^{-n}|S|$, it suffices to prove that 
\begin{align} \label{eq: appendix 3}
    \eqref{eq: appendix 2} \leq C(n,E,\eps) (\delta/\lambda)^{-\eps} |L_\lambda(S)| (\delta/\lambda)^{-(n-1)}.
\end{align}
The estimate \eqref{eq: appendix 3} is a consequence of Lemma \ref{lem: appendix 1} below applied with `$\lambda$' as $\lambda$ and `$\delta$' as $\delta/\lambda$. We will spend the rest of this section proving Lemma \ref{lem: appendix 1}. 

\begin{lemma} \label{lem: appendix 1}
    For every $\eps > 0$ and integer $E \geq 2$, there exists $C(n,E,\eps) > 0$ (which implicitly depends on derivatives of $\phi$) such that for any $\lambda \in (0,1]$, $\x_0 \in [-1,1]^{n}$, and $\delta \in (0,1]$ the following holds. For every $\delta$-direction separated collection of tubes $\T =\{T_{\xi_i,v_i,X_{\lambda,\x_0}}(\delta,1)\}$ (with $(\xi_i,v_i) \in [-1,1]^{2n-2}$), we have 
    \begin{align}
        |\{T \in \T : T \subset S\}| \leq C(n,E,\eps) \delta^{-\eps} |S| \delta^{-(n-1)}
    \end{align}
    whenever $S \subset B(0,10)$ is a semi-algebraic set of complexity $\leq E$.  
\end{lemma}

Lemma \ref{lem: appendix 1} follows from an application of \cite[Theorem 3.1]{HormanderDichotomy}, which we recall below.

\begin{theorem}[\protect{\cite[Theorem 3.1]{HormanderDichotomy}}]\label{thm: generalized PWA}
    Consider a map $\Phi : \R^{n-1} \times \R^{n-1} \times \R \to \R^{n-1}$ which is smooth on a neighborhood of $[-1,1]^{2n-1}$.  
    Suppose that for every choice of $(\xi,v) \in [-1,1]^{2n-2}$, 
    \begin{align}\label{eq: gen PWA lower bd}
        \int_{-\delta^\eps}^{\delta^\eps} |\det(\nabla_v\Phi(\xi,v,t)M + \nabla_\xi \Phi(\xi,v,t))|dt \gtrsim_\eps \delta^{C\eps}, \ \forall M \in \mathrm{Mat}_{(n-1) \times (n-1)}(\R)
    \end{align}
    for a constant $C$ that depends only on the dimension $n$, and the bound is uniform and independent of the choices of $v,\xi$ and $M$. Then for every collection $\T = \{T_{\xi_i,v_i,\Phi}(\delta,1)\}_i$ (with $(\xi_i,v_i)\in[-1,1]^{2n-2}$) of $\delta$-tubes in $\delta$-separated directions, 
    \begin{align}
        |\{T \in \T : T \subset S\}| \leq C(n,E,\eps) |S| \delta^{1-n-\eps}
    \end{align}
    whenever $S \subset B(0,10)$ is a semi-algebraic set of complexity $\leq E$. Moreover $C(n,E,\eps)$ only depends on upper bounds of finitely many (depending on $n,E,\eps$) derivatives of $\Phi$, and the implicit constant in \eqref{eq: gen PWA lower bd}. 
\end{theorem}

\begin{proof}[Proof of Lemma \ref{lem: appendix 1} by applying Theorem \ref{thm: generalized PWA}]
    Once we prove the following two uniform estimates, we can apply Theorem \ref{thm: generalized PWA} with $\Phi = X_{\lambda, \x_0}$ to find the uniform constant $C(n,E,\eps)$ that gives Lemma \ref{lem: appendix 1}.
    \begin{enumerate}
        \item[(i)] For all integers $k \geq 0$, $\|X_{\lambda, \x_0}\|_{C^k([-1,1]^{2n-1})} \lesssim_k 1$ uniformly over $\x_0 \in [-1,1]^n$ and $\lambda \in (0,1]$.
        \item[(ii)] For all $\eps > 0$ \eqref{eq: gen PWA lower bd} holds with $\Phi = X_{\lambda, \x_0}$, uniformly over $\x_0 \in [-1,1]^n$, $\lambda \in (0,1]$, $(\xi,v) \in [-1,1]^{2n-2}$, and $\delta \in (0, 1]$. 
    \end{enumerate}

    \noindent \textbf{Part (i)}
    As shorthand, we will write $w_0=v_{\x_0}(\xi)$, $w = v_{\x_0}(\xi) + \lambda v$, $t' = t_0+\lambda t$, and $\x = (X(\xi,v_{\x_0}(\xi)+ \lambda v, t_0 + \lambda t),t_0+\lambda t)$.
    We can compute 
    \begin{align}
        \nabla_\xi X_{\lambda, \x_0}(\xi,v,t) &= \lambda^{-1}(\nabla_\xi X(\xi,w,t') + \nabla_v X(\xi,w,t') \nabla_\xi v_{\x_0}(\xi)) 
        \\
        &= -\lambda^{-1} \nabla_x \nabla_\xi \phi (\x,\xi)^{-1}(\nabla_\xi^2 \phi(X(\xi,w,t'), t',\xi) -\nabla_\xi^2 \phi(X(\xi,w_0, t_0), t_0,\xi))\label{eq: D xi X expr 0} \\
        &=-\nabla_x \nabla_\xi \phi(\x,\xi)^{-1} \int_0^1 \nabla_{(v,t)} H_{\x_0}(\xi,v_{\x_0}(\xi)+\lambda sv, t_0+\lambda st)\cdot(v,t) ds, \label{eq: D xi X expr}
    \end{align}
    where we defined the smooth map
    \begin{align}
        H_{\x_0}(\xi,v,t) := \nabla_{\xi}^2\phi(X(\xi,v,t),t,\xi).
    \end{align}
    Clearly \eqref{eq: D xi X expr} is bounded uniformly, and its further derivatives are bounded, since $\phi$ is smooth and $(\nabla_x \nabla_\xi \phi)^{-1}$ is smooth on a neighborhood of $[-2,2]^{2n-1}$. Indeed, further derivatives produce only positive powers of $\lambda$ (or no powers of $\lambda$). We also have 
    \begin{align}
        \nabla_v X_{\lambda, \x_0}(\xi,v,t) &=
        \nabla_x \nabla_\xi \phi(\x,\xi)^{-1}
        \\
        \partial_t X_{\lambda,\x_0}(\xi,v,t)&= - \nabla_x \nabla_\xi\phi(\x,\xi)^{-1} \partial_t \nabla_\xi \phi(\x,\xi)
    \end{align}
    which are similarly bounded uniformly, and so are all their further derivatives. 
    
    \noindent \textbf{Part (ii)}

    This is almost identical to the proof of \cite[Theorem 1.2]{HormanderDichotomy}, but we can use Proposition \ref{prop: mild reformulation Bourgain} to simplify it somewhat. By Proposition \ref{prop: mild reformulation Bourgain}, we have 
    \begin{align}\label{eq: 100}
        \nabla_\xi^2 \phi(X(\xi,v,t),t,\xi) = A(v,\xi) +\tilde c(\xi,v,t) B(v,\xi),
    \end{align}
    where $\tilde c(\xi,v,t) = c(X(\xi,v,t),t,\xi)$. As in the proof of Theorem \ref{thm: geometric characterization of Bourgain}, $\partial_t \tilde c(\xi,v,t) \neq 0$ on $[-2,2]^{2n-1}$. We may then assume without loss of generality that $t \mapsto \tilde c(\xi,v,t)$
    is invertible, possibly after a constant rescaling of $t$ in $\phi(x,t,\xi)$.
    Also $\det(B(v,\xi)) \neq 0$ on $[-1,1]^{2(n-1)}$.  Substituting \eqref{eq: 100} into \eqref{eq: D xi X expr 0}, 
    \begin{align}
        \nabla_\xi X_{\lambda, \x_0}(\xi,v,t) 
        &= -\lambda^{-1} \nabla_x \nabla_\xi \phi(\x,\xi)^{-1}
        (\tilde A(w,\xi) + \tilde c(\xi,w,t')B(w,\xi)),
    \end{align}
    where 
    \begin{align}
        \tilde A(w,\xi) = A(w,\xi)-A(w_0,\xi)-\tilde c(\xi,w_0,t_0)B(w_0,\xi).
    \end{align}
    By using \cite[Lemma 3.2]{HormanderDichotomy} and the calculations above, we can 
    show \eqref{eq: gen PWA lower bd} holds uniformly for $X_{\lambda,\x_0}$.

    \begin{lemma}[\protect{\cite[Lemma 3.2]{HormanderDichotomy}}]\label{lem: elementary lin alg lemma}
        For $A,B \in \mathrm{Mat}_{k\times k}(\R)$ and a measurable $E \subset \R$, 
        \begin{align}
            \int_E |\det(tA+B)| dt \gtrsim_k |E|^{k+1}|\det(A)|. 
        \end{align}
    \end{lemma}
    
    By using $|\det(\nabla_x \nabla_\xi \phi)| \sim 1$, applying the change of variables $s = t_0 + \lambda t$, and then applying the change of variables  $r = \tilde c(\xi,w,s)$ (and using $|\partial_t \tilde c| \sim1$),
    \begin{align}
        &\int_{-\delta^\eps}^{\delta^\eps} |\det(\nabla_{v} X_{\lambda,\x_0}(\xi,v,t) M + \nabla_\xi X_{\lambda,\x_0}(\xi,v,t))| dt \\
        &\sim  \lambda^{-(n-1)} \int_{-\delta^\eps}^{\delta^\eps}
        |\det((\lambda M-\tilde A(w,\xi)) -\tilde c(\xi,w,t_0 + \lambda t)B(w,\xi))|dt \\
        &= \lambda^{-n} \int_{[t_0-\lambda \delta^\eps, t_0+\lambda\delta^\eps]} |\det (C(w,\xi) -\tilde c(\xi,w,s)B(w,\xi))| ds  \\
        &\sim \lambda^{-n} \int_{\tilde c(\xi,w,[t_0-\lambda \delta^\eps, t_0+\lambda\delta^\eps])} |\det(C(w,\xi)-rB(w,\xi))|dr,\label{eq: last thing!}
    \end{align}
    where we set $C(w,\xi) = \lambda M - \tilde A(w,\xi)$. We have 
    \begin{align}
        |\tilde c(\xi,w,[t_0-\lambda \delta^\eps, t_0+\lambda\delta^\eps])| \gtrsim |[t_0-\lambda \delta^\eps,t_0 + \lambda \delta^\eps]| \gtrsim \lambda \delta^\eps,
    \end{align}
    so we may apply Lemma \ref{lem: elementary lin alg lemma} to get 
    \begin{align}
        \eqref{eq: last thing!} \gtrsim \lambda^{-n} (\lambda \delta^\eps)^{n} |\det(B(w,\xi))| \gtrsim \delta^{n\eps}
    \end{align}
    uniformly in $\x_0 \in [-1,1]^n, \lambda\in (0,1], (\xi,v)\in [-1,1]^{2(n-1)} , \delta \in (0,1]$. 
\end{proof}

\subsection{Conjecture \ref{conj: PWA Kak conj} implies Conjecture \ref{conj: Bourgain Kakeya} in $\R^3$: proof of Proposition \ref{prop: PWA conj implies dir sep conj}}\label{subsec: PWA conj implies dir sep conj}

We recall Proposition \ref{prop: PWA conj implies dir sep conj} below. 

\propPWAconjimpliesdirsepconj*

\begin{proof}[Proof of Proposition \ref{prop: PWA conj implies dir sep conj}]
    Let $\phi$ be a phase function in $\R^3$ satisfying Bourgain's condition and fix $\eps \in (0,1)$. Define $\eps' = \eps/4$ and $\beta = \eps/40$, so $\eps' + 10\beta \leq \eps/2$.
    Applying Conjecture \ref{conj: PWA Kak conj} with this value of $\beta$, we obtain $E = E(\beta)$,
    $\eta' = \eta_{\PW}(\eps',E)$, $\delta'_0 = \delta_{0,\PW}(\eps',E)$, such that the following holds for all $\delta \in (0,\delta_0']$. If $B$ is a $\delta^\beta$-ball and $(\T^B, Y^B)$ is a set of $(\delta,\phi, B)$-tubelets with $\Delta^E_{\max}(\T^B) \leq \delta^{-\eta'}$ and $\delta^{\eta'}$-dense shadings, then $|\bigcup_{T^B \in \T^B} Y^B(T^B)| \geq \delta^{\eps'} |\T^B| |T^B|$.

    Begin by choosing $\delta_0 \leq \delta'_0$ to be decreased further during the proof. 
    We will prove the desired estimate with $\eta = \eta'/2$. Let $(\T,Y)$ be a $\delta$-direction separated collection of $(\delta, \phi)$-tubes with $\delta^\eta$-dense shadings. Consider a collection of $\delta^\beta$-balls $\mathcal B$ so that $\{(1/2)B\}_{B \in \mathcal B}$ is finitely overlapping and covers the ambient bounded region $M$. We have $|\mathcal B| \lesssim \delta^{-3\beta}$.

    If $T \cap (1/2)B \neq \emptyset$, let $T^B$ be the associated $(\delta, \phi, B)$-tubelet satisfying $T \cap B \subset T^B$ (we formally define these at the start of Appendix \ref{subsec: curved sticky PWA}).
    For each $T \in \T$, we can find $B(T) \in \mathcal B$ such that $|Y(T) \cap B(T)| \geq \delta^{\eta'}|T^{B(T)}|$. 
    Indeed, since $Y(T)$ is $\delta^\eta$-dense and $|\{B \in \mathcal B : T \cap (1/2)B \neq \emptyset \}| \lesssim \delta^{-\beta}$, we have 
    \begin{align}
        \delta^{\eta}|T| \lesssim \delta^{-\beta}\max_{B: T \cap (1/2)B \neq \emptyset} |Y(T) \cap B|. 
    \end{align}
    Choosing $B(T)$ achieving the maximum, rearranging, and using $\delta^\beta |T| \sim |T^B|$ finishes it. 

    Now we pigeonhole in $B \in \mathcal B$. Since $|\mathcal B| \lesssim \delta^{-3\beta}$, there is $B' \in \mathcal B$ such that 
    \begin{align}
        \T' := \{T \in \T : B(T) = B'\}
    \end{align}
    satisfies $|\T'| \gtrsim \delta^{3\beta}|\T|$. Refine $\T'$ to a $2\delta/\delta^{\beta}$-direction separated collection $\T''$, satisfying $|\T''|\gtrsim \delta^{5\beta}|\T|$. We define the set of tubelets
    $\T^{B'} = \{T^{B'} : T \in \T''\}$, and the $\delta^{\eta'}$-dense shadings $Y^{B'}(T^{B'}) = Y(T) \cap B' \subset T^{B'}$. 

    It remains to show that $\Delta_{\max}^E(\T^{B'}) \leq \delta^{-\eta'}$. 
    Fix $S \subset B'$ semi-algebraic of complexity $\leq E$. Theorem \ref{thm: dir sep implies PWA} applies to $\T''$ with `$\eps$' as $\eta$ and `$\lambda$' as $\delta^\beta/2$ to give
    \begin{align}
        |\T^{B'}[S]| &\leq |\{T \in \T'' : T \cap (1/2)B' \neq \emptyset \text{ and } T \cap B' \subset S\}| \\
        &\leq C(3,E,\eta) \delta^{-\eta} |S| (\delta^{2+\beta}/2)^{-1} \\
        &\leq \delta^{-\eta'}|S| |T^{B'}|^{-1},
    \end{align}
    with $\delta_0$ chosen small enough. Thus $\Delta(\T^{B'},S) \leq \delta^{-\eta'}$ and, after taking a supremum over $S$, $\Delta_{\max}^E(\T^{B'}) \leq \delta^{-\eta'}$. 
    Now $(\T^{B'},Y^{B'})$ satisfies the hypotheses in Conjecture \ref{conj: PWA Kak conj} and $|\T^{B'}| \gtrsim \delta^{5\beta}|\T|$, so we obtain 
    \begin{align}
        |\bigcup_{T \in \T} Y(T)| &\geq |\bigcup_{T^{B'} \in \T^{B'}} Y^{B'}(T^{B'})| \\
        &\geq \delta^{\eps'} |\T^{B'}| |T^{B'}| \\
        &\gtrsim \delta^{\eps' + 6\beta} |\T| |T| \\
        &\geq \delta^{\eps} |\T| |T|,
    \end{align}
    using that $\eps' + 6 \beta \leq \eps'+10\beta\leq \eps/2 < \eps$ and shrinking $\delta_0$. 
\end{proof}

\subsection{Multiscale Frostman polynomial Wolff estimate: proof of Theorem \ref{thm: curved sticky PWA}} \label{subsec: curved sticky PWA}

We recall Theorem \ref{thm: curved sticky PWA} below.
\thmcurvedstickyPWA*

Below we construct the semi-algebraic approximations of $\phi$-tubes through a ball, and polynomial approximations of the diffeomorphisms $F_{\xi_0,v_0}$ from Theorem \ref{thm: geometric characterization of Bourgain}. Fix $\beta \in (0,1)$ and a ball $B$ with center $\x_B = (x_B,t_B)$, and radius $\delta^\beta$ for $\delta \leq \delta_0$. Here $\delta_0$ is a number that depends only on $\beta$ and derivatives of $\phi$, and we will implicitly shrink $\delta_0$ to absorb various constants. Let $\rho \in [\delta,\delta^\beta]$ and consider a $(\rho,\phi)$-tube $T$ with coaxial $\phi$-curve $\ell_{\xi,v}$ passing through $(1/2)B$. Recall the parameterization $\ell_{\xi,v}(t) = (X(\xi,v,t), t)$. Let $t_B$ be the $t$-coordinate of the center of $B$. By Taylor expanding we have $X(\xi,v,t) = X^B(\xi,v,t)+ O(|t-t_B|^{5\beta^{-1}})$, where $t \mapsto X^B(\xi,v,t)$ is a polynomial of degree at most $5 \beta^{-1}$. For $|t - t_B| \leq \delta^{\beta}$, we certainly have $|X(\xi,v,t) - X^B(\xi,v,t)| \leq \delta/1000$. Define the $(\rho, \phi, B)$-tubelet associated to $T$ by 
\begin{align}
    T^B :&= \{(X^B(\xi,v,t) + y, t) : |t - t_B| \leq \delta^\beta, |y| \leq 2\rho\} \\
    &= \{(x,t) : (|t - t_B|^2 \leq (\delta^\beta)^2) \wedge (|x - X^B(\xi,v,t)|^2 \leq (2\rho)^2)\}.
\end{align}
Clearly $T\cap B \subset T^B$ and $T^B$ is semi-algebraic with complexity $\leq 100\beta^{-1}$. 

As in the proof of Theorem \ref{thm: sticky reduction formal}, we can assume all derivatives of $F_{\xi_0,v_0}$ are uniformly bounded. 
For each $(\xi_0,v_0) \in \Sigma_0 \times  \mathcal V_0$, Taylor expand $F_{\xi_0,v_0}$ to get $F_{\xi_0,v_0}(\x) = F^B_{\xi_0,v_0}(\x) + O(|\x-\x_B|^{5 \beta^{-1}})$, where $F^B_{\xi_0,v_0}$ is a polynomial of degree at most $5 \beta^{-1}$.
Let $\rho \in [2 \delta, 1/2]$ and set $\tau=\max(\rho^2,\delta)\leq \rho/2$, and fix a $(\rho, \phi, B)$-tubelet $T_{\rho}^B$ whose coaxial curve has parameters $(\xi_0,v_0)$. By using \eqref{eq: bourgain characterization}, for each $(\tau, \phi, B)$-tube $T_{\tau}^B \subset T^B_{\rho}$, there is a straight radius $\tau$ and length $\delta^\beta$ tubelet $T_{\tau}^{B,\lin}$ such that $F^B_{\xi_0,v_0}(T_{\tau}^B) \subset T_{\tau}^{B,\lin}$. The following lemma will help us pass from Frostman polynomial Wolff collections of tubes to Frostman convex Wolff collections of tubes. 
\begin{lemma}\label{lem: convex set estimate}
    Let $F : U \to \R^3$ be a polynomial map of degree at most $D$, which is a diffeomorphism onto its image. 
    Let $K$ be a bounded convex set whose outer John ellipsoid is contained in $F(U)$. Then there exists a semi-algebraic set $S$ of complexity at most $10D$ such that $K \subset F(S)$ and $|F(S)| \lesssim |K|$. 
\end{lemma}
\begin{proof}
    Let $C$ be the outer John ellipsoid of $K$, so $K \subset C$ and $|C| \lesssim |K|$. The ellipsoid $C$ takes the form $C = \{x : (x-a)^T A(x-a) \leq 1\}$,
    for $A$ a symmetric positive-definite matrix. Define $S = F^{-1}(C)$, so $K \subset F(S)$ and $|F(S)| \lesssim |K|$. To see that $S$ is semi-algebraic of complexity at most $10D$, consider the representation
    \begin{align}
        S = \{x : (F(x)-a)^TA(F(x)-a) \leq 1\}.
    \end{align}
\end{proof}

Next we recall \cite[Theorem 7.3(A)]{guth2026streamlinedproofkakeyaset}, which is the main input to Theorem \ref{thm: curved sticky PWA}. Let $\T^\lin$ be a collection of straight $\delta$-tubes.
For a convex set $K$, define
\begin{align}
    C_F(\T^\lin,K) = \frac{\sup_{K'\subset K} \Delta(\T^\lin,K')}{\Delta(\T^\lin,K)},
\end{align}
where the supremum is taken over $K' \subset K$ convex. 
We say $\T^\lin$ is \emph{$C$-Frostman convex Wolff} inside $K$ if $C_F(\T^\lin, K) \leq C$. Assume $\T^\lin$ is uniform. Then $\T^\lin$ is $C$-Frostman convex Wolff at every scale if for every $\delta \leq \rho \leq 1$, we have $C_F(\T^\lin[T_\rho], T_\rho)\leq C$ for every $T_\rho \in \T_\rho^\lin$.
\begin{theorem}[\protect{\cite[Theorem 7.3(A)]{guth2026streamlinedproofkakeyaset}}] \label{thm: WA straight Kakeya}
    For all $\eps > 0$, there exists $\eta=\eta_{\mathrm{CW}}(\eps)>0, \delta_0 =\delta_{0,\mathrm{CW}}(\eps)> 0$ so that the following holds for all $\delta \in (0,\delta_0]$. 
    
    Let $(\T^\lin,Y^\lin)$ be a uniform set of straight $\delta$-tubes with $\delta^\eta$-dense shadings. Then if $\T^\lin$ is $\delta^{-\eta}$-Frostman convex Wolff at every scale, 
    \begin{align}
        |\bigcup_{T^\lin \in \T^\lin} Y^\lin(T^\lin)| \geq \delta^\eps.
    \end{align}
\end{theorem}

We now set up a version of Theorem \ref{thm: curved sticky PWA} which works better with induction. Below we take $E = 10^5\beta^{-1}$ throughout. 
\begin{definition}[Assertion $\mathrm{SK}_{\PW}^\beta(\phi, \eps, \eta, \delta_0, \kappa)$]
    Fix $\beta \in (0,1)$, and  $\kappa \in (\beta,1]$, $\eps, \eta, \delta_0 > 0$.  
    We say that $\mathrm{SK}_{\PW}^\beta(\phi, \eps, \eta, \delta_0, \kappa)$ holds if for all $\delta \in (0,\delta_0]$ the following holds. Set $\tau = \delta^\kappa$. Let $B$ be a ball of radius $\delta^\beta$ and $(\T^B,Y^B)$ a uniform set of $(\tau, \phi, B)$-tubelets with $\tau^\eta$-dense shadings. If $\T^B$ is $(\tau^{-\eta},E)$-Frostman polynomial Wolff at every scale, then 
    \begin{align}
        |\bigcup_{T^B \in \T^B} Y^B(T^B)| \geq \delta^{\eps} |B|.
    \end{align}
\end{definition}

Fix $\phi$ satisfying Bourgain's condition. Our goal is to show that for every $\eps > 0$ and $\beta \in (0,1)$, there are $\eta, \delta_0 > 0$ such that $\mathrm{SK}_{\PW}^\beta(\phi, \eps, \eta, \delta_0, 1)$ holds. We have the following base case.

\begin{lemma}\label{lem: PW base case}
    Fix $\beta \in (0,1)$ and $\eps > 0$. There exists $\eta = \eta_{\mathrm{PW,base}}(\beta,\eps)$, $\delta_0=\delta_{0, \mathrm{PW,base}}(\beta,\eps)$, and $\kappa = \kappa_{\mathrm{PW,base}}(\beta,\eps) \in (\beta,1)$ such that $\mathrm{SK}^\beta_{\PW}(\phi, \eps, \eta, \delta_0, \kappa)$ holds. 
\end{lemma}

\begin{proof}
    Fix $(\T^B,Y^B)$ satisfying the hypothesis of $\mathrm{SK}^\beta_{\PW}(\phi, \eps, \eta, \delta_0, \kappa)$, with $\eta,\delta_0,\kappa$ to be determined. By extracting a single tube $T^B \in \T^B$, 
    \begin{align}\label{eq: PWA base}
        |\bigcup_{T^B \in \T^B} Y^B(T^B)| \geq |Y^B(T^B)| \gtrsim \delta^{\eta \kappa + 2\kappa + \beta}.
    \end{align}
    Choose $\kappa \in (\beta,1)$ close enough to $\beta$ that $2(\kappa - \beta) < \eps/2$. Then choose $\eta >0$ small enough that $\eta \kappa < \eps/2 - 2(\kappa -\beta)$. Then we get $\eta \kappa + 2\kappa + \beta \leq \eps/2 + 3\beta$. By choosing $\delta_0$ small enough, the right hand side of \eqref{eq: PWA base} is $\geq \delta^{\eps} |B|$.  
\end{proof}

We have the following inductive step.

\begin{lemma}\label{lem: PWA induction}
    Let $\phi$ satisfy Bourgain's condition. Then for every $\beta, \eps, \eta, \delta_0$, $\kappa \in (\beta, 1)$, there exists $\tilde \eta = \eta_{\mathrm{PW,ind}}(\eta, \kappa) >0$ and $\tilde \delta_0 = \delta_{0,\mathrm{PW,ind}}(\eta, \delta_0, \kappa) > 0$ such that 
    \begin{align}
        \mathrm{SK}^\beta_{\PW}(\phi, \eps, \eta, \delta_0, \kappa) \implies \mathrm{SK}^\beta_{\PW}(\phi, \eps, \tilde \eta, \tilde \delta_0, \min(2\kappa,1)).
    \end{align}
\end{lemma}

Theorem \ref{thm: curved sticky PWA} follows from Lemmas \ref{lem: PW base case} and \ref{lem: PWA induction} by a straightforward iteration, which is essentially the same as in the proof of Theorem \ref{thm: sticky reduction formal}. We therefore omit the details and proceed to proving Lemma \ref{lem: PWA induction}.

\begin{proof}[Proof of Lemma \ref{lem: PWA induction}]
In the proof, we will use \emph{refinements} in the sense of \cite[Section 2]{guth2026streamlinedproofkakeyaset}.

\noindent \textbf{Step 1: Setup}

    Apply Theorem \ref{thm: WA straight Kakeya} to obtain the constants $\eta'=\eta_{\mathrm{CW}}(\eta/100)$ and $\delta_0'=\delta_{0,\mathrm{CW}}(\eta/100) > 0$. Define $\tilde \eta = (1-\kappa)\min(\eta, \eta')/100$, and begin by assuming $\tilde \delta_0 \leq \min(\delta_0, \delta_0')$. Throughout the argument we will implicitly shrink $\tilde \delta_0$ to absorb various constants.
    Set $\rho = \delta^{\kappa}$ and $\tau = \delta^{\min(2\kappa,1)} = \max(\rho^2, \delta)$. Let $B$ be a ball of radius $\delta^\beta$ and $(\T^B,Y^B)$ a uniform set of $(\tau,\phi,B)$-tubelets with $\tau^{\tilde \eta}$-dense shadings and $\T^B$ $(\tau^{-\tilde \eta},E)$-Frostman polynomial Wolff at every scale. We now study $\T^B[T^B_\rho]$ for a fixed $T^B_\rho \in \T^B_\rho$.

    \noindent \textbf{Step 2(a): Straightening and rescaling $\T^B[T^B_\rho]$ to $\T^{\lin}$}

    Let $(\xi_0,v_0)$ be the parameters of the coaxial $\phi$-curve of $T_\rho^B$ and set $F^B = F^B_{\xi_0,v_0}$. For each $T^B \in \T^B[T^B_\rho]$, there is a straight radius $\tau$ length $\delta^\beta$ tubelet $T_\tau^{B,\lin}$ such that $F^B(T^B) \subset T^{B,\lin}$. The associated shading is $Y^{B,\lin}(T^{B,\lin})=F^B(Y^B(T^B))$. Let $\T^{B,\lin}$ be the set of these straight tubelets. Also let $T^{B,\lin}_\rho$ be a straight radius $\rho$ length $\delta^\beta$ tubelet such that $F^B(T^{B}_\rho) \subset T_\rho^{B,\lin}$. By composing $F^B$ with a rotation and translation, we may further assume that 
    \begin{align}
        T_\rho^{B,\lin} = \{(y,t) : |t| \leq \delta^\beta, |y| \leq 2\rho\}.
    \end{align}
    Define the linear map $L(x',t) = (\rho^{-1} x',\delta^{-\beta}t)$, which sends $T_{\rho}^{B,\lin}$ to (approximately) a unit ball. Finally define the set of straight $(\tau/\rho)$-tubes $\T^{\lin} = \{L(T^{B,\lin}) : T^{B,\lin} \in \T^{B,\lin}\}$, with the associated shadings $Y^{\lin}(T^\lin_{\tau/\rho}) = L(Y^{B,\lin}(T^{B,\lin}))$. By a harmless refinement, we may assume that $(\T^\lin, Y^\lin)$ is uniform. 

    \noindent \textbf{Step 2(b): Checking $\T^{\lin}$ is Frostman convex Wolff at every scale}

    We need to show $\T^\lin$ is $(\tau/\rho)^{-{\eta'}}$-Frostman convex Wolff at every scale. That is, for each $\rho' \in [\tau/\rho, 1]$, we must show 
    $C_F(\T^\lin[T_{\rho'}^{\lin}], T_{\rho'}^\lin) \leq (\tau/\rho)^{-\eta'}$.
    It is enough to show 
    \begin{align}\label{eq: goal 1}
        C_F(\T^\lin[T_{\rho'}^\lin], T_{\rho'}^\lin) \lesssim \PW^E_{F}(\T^B[T^B_{\rho \rho'}], T_{\rho\rho'}^B).
    \end{align}
    Indeed, since $\T^B$ is $(\tau^{-\tilde \eta},E)$-Frostman polynomial Wolff at every scale we have 
    \begin{align}
        \PW_F^E(\T^B[T_{\rho\rho'}^B], T^B_{\rho\rho'}) \leq \tau^{-\tilde \eta} \leq C^{-1}(\tau/\rho)^{-\eta'},
    \end{align}
    where $C$ is the implicit constant in \eqref{eq: goal 1}. To see this, the inequality $\tau^{-\tilde \eta} \leq C^{-1}(\tau / \rho)^{-\eta'}$ is equivalent to 
    \begin{align}
        C \leq \delta^{\tilde \eta \min(2\kappa,1) - \eta'(\min(2\kappa,1)-\kappa)},
    \end{align}
    which certainly holds for small enough $\tilde \delta_0$ if
    \begin{align}\label{eq: 1 sec PWA}
        \tilde \eta \min(2\kappa,1) \leq \eta'(\min(2\kappa,1)-\kappa)/50
    \end{align}
    For $\kappa \in (0,1)$ we have $(\min(2\kappa,1)-\kappa)/\min(2\kappa,1) \geq (1-\kappa)/2$, and by definition $\tilde \eta \leq (1-\kappa)\eta'/100$. Thus \eqref{eq: 1 sec PWA} holds. 
    For the remainder of Step 2(b), we show \eqref{eq: goal 1}.     
    The denominator is straightforward to understand. We have 
    \begin{align} \label{eq: comparison 2}
        \Delta(\T^\lin[T_{\rho'}^\lin], T_{\rho'}^\lin) = \frac{|T^\lin| |\T^\lin[T_{\rho'}^\lin]|}{|T_{\rho'}^\lin|} \sim \frac{|T^B| |\T^B[T^B_{\rho\rho'}]|}{|T^B_{\rho\rho'}|} =\Delta(\T^B[T_{\rho\rho'}^B], T^B_{\rho\rho'}).
    \end{align}
    We now consider the numerator. Fix $K' \subset T_{\rho'}^\lin$ convex. Then 
    $L^{-1}(K') \subset T_{\rho\rho'}^{B,\lin}$. By applying Lemma \ref{lem: convex set estimate} to the polynomial diffeomorphism $F^B$ of degree at most $5\beta^{-1}$ and to the convex set $L^{-1}(K')$, we find a semi-algebraic set $S$ of complexity at most $50\beta^{-1}$ such that $L^{-1}(K') \subset F^B(S)$ and $|F^B(S)| \lesssim |L^{-1}(K')|$. Define $S' = S \cap T_{\rho\rho'}^B \subset T_{\rho\rho'}^B$, which is a semi-algebraic set of complexity at most $1000\beta^{-1}$. 
    A tube $T^{\lin} \in (\T^\lin[T_{\rho'}^\lin])[K']$ comes from $T^{B,\lin} \in \T^{B,\lin}[T_{\rho\rho'}^{B,\lin}][L^{-1}(K')]$, which further comes from $T^B \in \T^B[T_{\rho\rho'}^B]$ satisfying $F^B(T^B) \subset T^{B,\lin} \subset L^{-1}(K') \subset F^B(S)$. Thus $T^B \subset (\T^B[T_{\rho\rho'}^B])[S']$, and we obtain 
    \begin{align} \label{eq: comparison 1}
        |(\T^\lin[T_{\rho'}^\lin])[K']| \leq |(\T^B[T_{\rho\rho'}^B])[S']|. 
    \end{align}
    Using that $F^B$ has Jacobian $\sim 1$ and $L^{-1}$ has Jacobian $\sim \rho^2 \delta^\beta$, we get $|K'| \gtrsim \rho^{-2} \delta^{-\beta} |S'|$. Combining this estimate and \eqref{eq: comparison 1}, we get 
    \begin{align}\label{eq: Delta est}
        \Delta(\T^{\lin}[T_{\rho'}^\lin], K') \lesssim \Delta(\T^B[T_{\rho\rho'}^B], S').
    \end{align}
    By \eqref{eq: comparison 2} and \eqref{eq: comparison 1},
    \begin{align}
        \frac{\Delta(\T^\lin[T_{\rho'}^\lin], K')}{\Delta(\T^\lin[T_{\rho'}^\lin], T_{\rho'}^\lin)} \leq \frac{\Delta(\T^B[T_{\rho\rho'}^B], S')}{\Delta(\T^B[T_{\rho\rho'}^B], T_{\rho\rho'}^B)}.
    \end{align}
    Taking a supremum over $S' \subset T_{\rho\rho'}^B$ of complexity $\leq E$ and then $K' \subset T_{\rho'}^\lin$ convex, we obtain \eqref{eq: goal 1}.

    \noindent \textbf{Step 2(c): Applying Theorem \ref{thm: WA straight Kakeya} to $(\T^\lin, Y^\lin)$}

    We can now apply Theorem \ref{thm: WA straight Kakeya} to $(\T^\lin,Y^\lin)$ to obtain 
    \begin{align}
        |\bigcup_{T^B\in\T^B[T_\rho^B]} L(F^B(Y^B(T^B))| =|\bigcup_{T^\lin \in \T^\lin} Y^\lin(T^\lin)| \geq (\tau/\rho)^{\eta/100}.
    \end{align}
    Since $L \circ F^B$ has Jacobian $\sim \rho^{-2} \delta^{-\beta}$, 
    \begin{align}\label{eq: PW shading ineq}
        |\bigcup_{T^B \in \T^B[T_\rho^B]} Y^B(T^B)| \gtrsim \rho^{2} \delta^\beta (\tau /\rho)^{\eta/100} \geq \rho^{\eta/2} |T_\rho^B|.
    \end{align}
    
    \noindent \textbf{Step 3: Applying $\mathrm{SK}_{\mathrm{PW}}^\beta(\phi, \eps, \eta,\delta_0, \kappa)$ and conclusion}

    By \eqref{eq: PW shading ineq}, for each $T_\rho^B \in \T^B_\rho$ we 
    get $\rho^{\eta/2}$-dense shadings given by 
    \begin{align}
        Y^B_\rho(T^B_\rho) = \bigcup_{T^B \in \T^B[T_\rho^B]} Y^B(T^B).
    \end{align}
    Just as in \cite[Remark 3.3]{guth2026streamlinedproofkakeyaset}, the property of being $(C,E)$-Frostman polynomial Wolff is inherited upwards. For each $\rho' \in [\rho, \delta^\beta]$, we know $\T^B[T_{\rho'}^B]=\bigsqcup_{T^B_\rho\in \T^B_{\rho}[T_{\rho'}^B]} \T^B[T_\rho^B]$ is $(\tau^{-\tilde \eta},E)$-Frostman polynomial Wolff in $T^B_{\rho'}$.
    Since $\T^B$ is also uniform, we conclude that $\T_\rho^B[T_{\rho'}^B]$ is $(\lesssim \tau^{-\tilde \eta},E)$-Frostman polynomial Wolff in $T_{\rho'}^B$. Thus using that $\tau^{-\tilde \eta} \leq \rho^{-\eta/4}$, we find that $\T_\rho^B$ is $(\rho^{-\eta/2},E)$-Frostman polynomial Wolff at every scale. After a refinement, we may assume that $(\T_\rho^B,Y_\rho^B)$ is a uniform set of tubelets and shadings, with $\rho^{\eta}$-dense shadings, and $\T^B_\rho$ $(\rho^{-\eta},E)$-Frostman polynomial Wolff at every scale. Applying $\SK_{\PW}^\beta(\phi,\eps, \eta, \delta_0,\kappa)$ to $(\T_\rho^B, Y_\rho^B)$, we conclude 
    \begin{align}
        |\bigcup_{T^B \in \T^B} Y^B(T^B)| \geq |\bigcup_{T_\rho^B \in \T_\rho^B} Y^B_\rho(T_\rho^B)| \geq \delta^\eps |B|.
    \end{align}
    
\end{proof}

\bibliographystyle{alpha}
\bibliography{reference}
	
\end{document}